\documentclass[10pt]{amsart}
\usepackage{amssymb,amsmath,amsfonts,amscd,euscript}

\newcommand{\nc}{\newcommand}
\usepackage{xcolor}
\usepackage{hyperref}
\numberwithin{equation}{section}
\newtheorem{thm}{Theorem}[section]
\newtheorem*{theorem}{Main Theorem}
\newtheorem{prop}[thm]{Proposition}
\newtheorem{lem}[thm]{Lemma}
\newtheorem{cor}[thm]{Corollary}
\newtheorem*{coro}{Corollary}
\newtheorem*{rem*}{Remark}
\theoremstyle{remark}
\newtheorem{rem}[thm]{Remark}
\newtheorem{definition}[thm]{Definition}
\newtheorem{example}[thm]{Example}
\newtheorem{dfn}[thm]{Definition}

\nc{\gl}{\mathfrak{gl}}
\nc{\GL}{\mathfrak{GL}}
\nc{\g}{\mathfrak{g}}
\nc{\gh}{\widehat\g}
\nc{\fh}{\mathfrak{h}}
\nc{\fu}{\mathfrak{u}}
\nc{\la}{\lambda}
\nc{\al}{\alpha }
\nc{\be}{\beta }
\nc{\ve}{\varepsilon }
\nc{\om}{\omega }

\nc{\ta}{\theta}
\nc{\veps}{\varepsilon}
\nc{\ch}{{\mathop {\rm ch}}}
\nc{\Tr}{{\mathop {\rm Tr}\,}}
\nc{\Id}{{\mathop {\rm Id}}}
\nc{\ad}{{\mathop {\rm ad}}}
\nc{\bra}{\langle}
\nc{\ket}{\rangle}
\nc{\x}{{\bf x}}
\nc{\bs}{{\bf s}}
\nc{\bz}{{\bf z}}
\nc{\bp}{{\bf p}}
\nc{\bfp}{{\mathbf p}}
\nc{\bfx}{{\mathbf x}}
\nc{\bfq}{{\bf q}}
\nc{\bq}{{\bf q}}
\nc{\br}{{\bf r}}
\nc{\bc}{{\bf c}}
\nc{\bt}{{\bf t}}
\nc{\bff}{{\bf f}}
\nc{\pa}{\partial}
\nc{\ld}{\ldots}
\nc{\cd}{\cdots}
\nc{\hk}{\hookrightarrow}
\nc{\T}{\otimes}
\nc{\gr}{\mathrm{gr}}
\nc{\ov}{\overline}

\nc{\cO}{\mathcal O}
\nc{\msl}{\mathfrak{sl}}
\nc{\msp}{\mathfrak{sp}}
\nc{\mgl}{\mathfrak{gl}}
\nc{\cS}{\mathcal S}
\nc{\U}{\mathrm U}
\nc{\V}{\EuScript V}
\nc{\bH}{\EuScript H}
\nc{\Res}{\mathrm{Res\ }}
\newcommand{\ol}{\overline}

\newcommand{\bC}{{\mathbb C}}
\newcommand{\bU}{{\mathbb U}}

\newcommand{\bV}{{\mathbb V}}
\newcommand{\bR}{{\mathbb R}}
\newcommand{\bZ}{{\mathbb Z}}

\newcommand{\bN}{{\mathbb N}}
\newcommand{\bP}{{\mathbb P}}
\newcommand{\bG}{{\mathbb G}}

\newcommand{\fg}{{\mathfrak g}}

\newcommand{\fn}{{\mathfrak n}}

\newcommand{\Fl}{\EuScript{F}}

\begin{document}

\title[Favourable modules]
{Favourable modules: Filtrations, polytopes, Newton-Okounkov bodies and flat degenerations}

\author{Evgeny Feigin, Ghislain Fourier and Peter Littelmann}
\address{Evgeny Feigin:\newline
Department of Mathematics,\newline
National Research University Higher School of Economics,\newline
Vavilova str. 7, 117312, Moscow, Russia,\newline
{\it and }\newline
Tamm Theory Division, Lebedev Physics Institute
}
\email{evgfeig@gmail.com}
\address{Ghislain Fourier:\newline
Mathematisches Institut, Universit\"at zu K\"oln,\newline
Weyertal 86-90, D-50931 K\"oln,Germany\newline
{\it and }\newline
School of Mathematics and Statistics\newline
University of Glasgow\\
UK
}
\email{gfourier@math.uni-koeln.de, Ghislain.Fourier@glasgow.ac.uk}
\address{Peter Littelmann:\newline
Mathematisches Institut, Universit\"at zu K\"oln,\newline
Weyertal 86-90, D-50931 K\"oln, Germany
}
\email{peter.littelmann@math.uni-koeln.de}

\begin{abstract}
We introduce the notion of a favourable module for a complex unipotent algebraic group, whose properties are governed
by the combinatorics of an associated polytope. We describe two filtrations of the module,
one given by the total degree on the PBW basis of the corresponding Lie algebra, the other by
fixing a homogeneous monomial order on the PBW basis.

In the favourable case a basis of the module is parametrized by the lattice points of a normal polytope. The
filtrations induce flat degenerations of the corresponding flag variety to its abelianized version and to a toric variety, the
special fibres of the degenerations being projectively normal and arithmetically Cohen-Macaulay. The
polytope itself can be recovered as a Newton-Okounkov body. We conclude the paper by giving classes of examples
for favourable modules.
\end{abstract}

\maketitle

\section*{Introduction}
Let $\bU$ be a complex algebraic unipotent group acting on a cyclic finite dimensional complex vector space $M$,
so for the nilpotent Lie algebra ${\mathfrak N}=\text{Lie\,}\bU$  and a cyclic vector $v_M$ we have $M=\U({\mathfrak N})v_M$.
A well known  example we have in mind is a maximal unipotent subgroup $\bU$ of a reductive algebraic group $G$
acting on an irreducible finite dimensional representation of $G$.

We call such a module $M$ {\it favourable} if important properties of the module are governed by
polytope combinatorics. More precisely, starting with an ordered basis $\{f_1,\dots, f_N\}$  of ${\mathfrak N}$
and an induced homogeneous monomial ordering ``$\le$" on the monomials in $U({\mathfrak N})$, consider the induced
filtration of $M$ defined by
$$
M_{\bq}=\text{$\bC$-span of }\{ \bff^\bp v_M=f_1^{p_1}\dots f_N^{p_N}v_M \text{\ for all } \bff^\bp\le\bff^\bq\}
$$
where $\bp,\bq\in \bN^{N}$ are multi-indices. In the associated graded module $\text{gr\,}^t M$ every homogeneous
component $\text{gr\,}^t M(\bp)$ is at most one-dimensional. Following Vinberg,
we call a monomial $\bff^\bp\in U({\mathfrak N})$
{\it essential for $M$} if $\text{gr\,}^t M(\bp)$ is non-zero, the exponent $\bp$ is called an essential multi-index for $M$.
The set $\text{es}(M)$ of all essential multi-indices is a finite subset of $\bZ^N$.

The first condition for $M$ to be favourable is that there exists a normal lattice 
polytope $P(M)\subset \bR^N$ such that
the lattice points $S(M)$ in $P(M)$ are exactly $\text{es}(M)$.\\ 
Recall the definition of
the Cartan component in the $n$-fold tensor product of $M$:
$$
M^{\odot n}=\U({\mathfrak N})(v_M\otimes \cdots \otimes  v_M)\subseteq M^{\otimes n}.
$$
The second condition is: $\dim M^{\odot n}=\sharp nS(M)$ for all $n\in\bN$,
it concerns the comparison of the number of points in the Minkowski sum $nS(M)$ with the dimension of
$M^{\odot n}$.

Recall the PBW-filtration of $\U({\mathfrak N})$ given by the span of all monomials up to a fixed total degree.
Since $M$ is cyclic, we get a natural induced filtration on $M$, which is coarser than the filtration above. The associated
graded space is denoted by $\text{gr}^a\, M$. Note that $\text{gr}^a\, M$ and $\text{gr}^t\, M$ are not anymore
$U({\mathfrak N})$-modules, but they are cyclic modules with generator $v_{M^a}$ respectively $v_{M^t}$ for the commutative algebra
$U({\mathfrak N}^a)$. Here ${\mathfrak N}^a$ is the abelian Lie
algebra with the same underlying vector space as $\mathfrak N$. Similarly on the group level, we have a commutative
unipotent group $\bU^a$ with Lie algebra ${\mathfrak N}^a$ acting on $\text{gr}^a\, M$ and $\text{gr}^t\, M$.

%\text{gr}^a\,

\begin{theorem}
Let $M$ be a favourable $\bU$-module and let  $P(M)$ be the associated normal lattice polytope.
\begin{enumerate}
%\item $P(M)$ is a normal polytope, i.e. $P(M)$ is an integer lattice polytope such that
%the set of lattice points in $nP(M)$ coincides with $nS(M)$.
\item 
Let $S(M)$  be the lattice points in $P(M)$.
The set $\{ \bff^\bp v_{M}^{\otimes n}\mid \bp\in nS(M)\}$ is a basis for $M^{\odot n}$ as well as for its
graded versions $\text{gr}^a\,M^{\odot n}$ and $\text{gr}^t\,M^{\odot n}$.
\end{enumerate}
\end{theorem}
To an action of a unipotent group, we associate a projective variety, which we
call a flag variety by analogy with the classical highest weight orbits of reductive
groups:
$$
\Fl_{\bU}(M)=\overline{\bU.[v_M]}\subseteq \bP(M),\quad
\Fl_{\bU^a}(\text{gr}^a\,M)=\overline{\bU^a.[v_{M^a}]}\subseteq \bP(\text{gr}^a\,M)
$$
respectively $\Fl_{\bU^a}(\text{gr}^t\,M)=\overline{\bU^a.[v_{M^t}]}\subset \bP(\text{gr}^t\,M)$.
\begin{theorem}
Let $M$ be a favourable $\bU$-module and let  $P(M)$ be the associated normal lattice polytope.
\begin{enumerate}
\item[{(ii)}] ${\Fl}_{\bU^a}(\text{gr}^t\,M)\subseteq \bP(\text{gr}^t\,M)$ is the toric variety defined by the polytope $P(M)$.
\item[{(iii)}] There exists a flat degeneration of ${\Fl}_{\bU}(M)$ into
${\Fl}_{\bU^a}(\text{gr}^a\,M)$, and for both there exists a flat degeneration
into ${\Fl}_{\bU^a}(\text{gr}^t\,M)$.
\item[{(iv)}] The projective flag varieties ${\Fl}_{\bU}(M)\subseteq \bP(M)$, 
${\Fl}_{\bU^a}(\text{gr}^a\,M)\subseteq \bP(\text{gr}^a\,M)$ and its toric version
${\Fl}_{\bU^a}(\text{gr}^t\,M)\subseteq \bP(\text{gr}^t\,M)$ are projectively normal and arithmetically Cohen-Macaulay
varieties.
%\item[{(v)}] For all $n\ge 1$ one has the following isomorphisms: $\Fl_{\bU}(M^{\odot n})\simeq \Fl_{\bU}(M)$,
%$\Fl_{\bU^a}(\text{gr}^a\, M^{\odot n})\simeq \Fl_{\bU^a}(\text{gr}^a\,M)$
%and $\Fl_{\bU^a}(\text{gr}^t\,M^{\odot n}) \simeq \Fl_{\bU^a}(\text{gr}^t M)$.
\end{enumerate}
\end{theorem}
By construction, $\bU$ as well as its abelianized version $\bU^a$ has a dense and open orbit in the
corresponding flag varieties.  
For a projective variety $X$ (and a fixed valuation $\nu$ of its
function field), let  $\Delta(X)$ denote its corresponding Newton-Okounkov body (for details
see Section~\ref{NObody}).
\begin{theorem} Let $M$ be a favourable $\bU$-module and let  $P(M)$ be the associated normal lattice polytope.
Let $v_M$ be a cyclic generator and let $\bV\subset \bU$ be the stabilizer
of $[v_M]\in \mathbb P(M)$. If $\bU$ satisfies some mild conditions (see Section~\ref{NObody}) and $M$ is favourable, then we have for the Newton-Okounkov bodies:
\begin{enumerate}
\item[{(v)}] The polytope $P(M)$ is the Newton-Okounkov body for the flag variety, its abelianized version and its toric version,
i.e. 
$$
\Delta(\Fl_{\bU}(M))=P(M)=\Delta(\Fl_{\bU^a}(\text{gr}^a\, M)).
$$
\end{enumerate}
\end{theorem}
Our main example and motivation for the study of these polytopes is our ongoing research on PBW-filtrations
and the associated degenerate flag varieties for the classical algebraic groups. Let $G$ be a simply connected
simple algebraic group with Lie algebra $\fg$. Fix a Cartan decomposition $\fg=\fn^-\oplus\fh\oplus \fn$, and
let $\bU$ be the maximal unipotent subgroup of $G$ with Lie algebra ${\mathfrak N}=\fn^-$.
As an immediate consequence of the results in \cite{FFoL1,FFoL2,G} we see that
\begin{coro}
For $G=SL_n(\bC)$, $G=Sp_{2n}(\bC)$ and $G=G_2$, there exists an ordering of the positive roots and a homogeneous
ordering on the monomials in $U(\mathfrak N)$ such that all irreducible finite dimensional $G$-modules are
favourable for $\bU$ with a highest weight vector as cyclic generator.
\end{coro}
The projective normality and the Cohen-Macaulay property of the flag variety
${\Fl}_{\bU^a}(\text{gr}^a\,M)$
were proved in \cite{FF} and \cite{FFiL} for $G=SL_n(\bC),Sp_{2n}(\bC)$ using an explicit desingularization;
the same result can be derived via the realization of the degenerate flag varieties in types $A$ and $C$ as Schubert varieties
\cite{CL}, \cite{CLL}.
For an explicit description of the corresponding polytopes and more examples see Section~\ref{flags}, where we 
also discuss the relation to Gelfand-Tsetlin and other string polytopes. 

In Section~\ref{orbit} we describe the coordinate ring of the flag varieties.
In Section~\ref{reps} we recall some generalities about filtrations and introduce the fundamental notions,
in Section~\ref{cr} we discuss the connection between filtrations and degenerations of flag varieties to toric varieties.
In Section~\ref{toricsection} we make a link to the toric geometry and in Sections~\ref{hcr1} and \ref{hcr2} we study the homogeneous coordinate rings.
In Section~\ref{favourablemodule} we introduce the notion of a favourable module and prove some first properties,
in Section~\ref{flatdegeneration} we show part {(iv)} %and {(v)} 
of the Main Theorem.  In Section~\ref{NObody} we show part {(v)} of the Main Theorem.
In Section~\ref{multiconeversion} we discuss a multi-cone version of the Main Theorem.
Examples are discussed in Section~\ref{flags}.

%%%%%%%%%%%%%%%%%%%%%%%%%%%%%%%%%%%%%%%%%%%%%%%%%%%%%%%%%%%%%%%%%%%%%%%%%%%%%%%%%%%%%%%%%%%%%%%%%%%%%%%%%%%%%%%%%%%%%%%%%%%%%%%%%
\section{The coordinate ring of an orbit closure}\label{orbit}
Let $M$ be a cyclic finite dimensional module for a connected complex algebraic group $H$.
Fix a cyclic generator $m_0\in M$. In the following we identify the \hbox{$\ell$-th} symmetric power 
$S^\ell(M)$ of $M$ with the symmetric tensors in the $\ell$-fold tensor product of $M$.
Another way to identify $S^\ell(M)$ with a subspace of $M^{\otimes \ell}$ is to view $S^\ell(M)$ as the linear
span $\langle GL(M). m_0^{\otimes \ell}\rangle \subset M^{\otimes \ell}$ of the orbit $GL(M). m_0^{\otimes \ell}$.

By the Cartan component $M^{\odot \ell}$ of the $\ell$-th tensor product $M^{\otimes \ell}$ we mean
the $H$-submodule 
$$
M^{\odot \ell}= \langle H. m_0^{\otimes \ell}\rangle\subseteq  \langle GL(M). m_0^{\otimes \ell}\rangle=S^\ell(M) \subseteq M^{\otimes \ell}.
$$
Note that $M^{\odot \ell}$ is a subspace of $S^\ell(M)$.

\subsection{The \texorpdfstring{$\ell$}{l}-tuple embedding}
We have a natural $GL(M)$-equivariant map of degree $\ell$:
$$
\phi: M\rightarrow S^\ell(M)\subset M^{\otimes \ell},\quad m\mapsto m^{\otimes \ell},
$$
which induces the $\ell$-tuple embedding:
$$
\bar\phi: \mathbb{P}(M)\rightarrow \mathbb{P}(S^\ell(M)),\quad [m]\mapsto [m^{\otimes \ell}].
$$
Next consider the varieties $X_1:=\overline{H.[m_0]}\subseteq \mathbb{P}(M)$ and
$X_\ell:=\overline{H.[m_0^{\otimes\ell}]}\subseteq \mathbb{P}(S^\ell(M))$. The map $\bar\phi$
is $GL(M)$-equivariant and an  isomorphism onto the image, so 
$$
\bar\phi(X_1)=\bar\phi(\overline{H.[m_0]})=\overline{H.[m_0]^{\otimes\ell}}=X_\ell.
$$
The orbit closure $X_\ell$ is nothing but the image of $X_1$ with respect to the $\ell$-tuple embedding.
Let  $\hat X_\ell\subseteq S^\ell(M)$ be the affine cone over $X_\ell$, then %, by the above,
%the $U$-Cartan component $M^{\odot \ell}\subset  S^\ell(M)$ is nothing but the linear span of $\hat X_\ell$
$$
M^{\odot \ell}=\langle \hat X_\ell\rangle \subseteq S^\ell(M).
$$
\subsection{Coordinate rings}
Let $\mathbb C[M]$ be the ring of polynomial functions. Endowed with the standard grading
$\mathbb C[M]=\bigoplus_{p\ge 0} \mathbb C[M]_p$ we view the ring also as the homogeneous
coordinate ring of $\mathbb P(M)$.

Similarly, let $\mathbb C[X_1]$ be the coordinate ring of the affine cone $\hat X_1$ over $X_1$. Endowed with the standard grading
$\mathbb C[X_1]=\bigoplus_{p\ge 0} \mathbb C[X_1]_p$ we view the ring as the homogeneous coordinate
ring of the embedded variety $X_1\hookrightarrow \mathbb P(M)$.

We %set $M^{\odot \ell}=M\odot\dots \odot M$ ($\ell$ times), and 
let $R_\ell(M)=(M^{\odot \ell})^*$
be the dual space. The vector space
\begin{equation}\label{gradingtotaldegree}
R(M)=\bigoplus_{\ell\ge 0} R_\ell(M)
\end{equation}
is then naturally endowed with a ring structure, where the multiplication maps $R_n\T R_k\to R_{n+k}$ are induced by the embeddings
$$
M^{\odot (n+k)}\hookrightarrow M^{\odot n}\T M^{\odot k}.
$$
\begin{prop}\label{repdescription}
$\mathbb C[X_1]\simeq R(M)$.
\end{prop}
\begin{proof}
Recall that the $\ell$-tuple embedding induces an isomorphism:
$$
\phi^*_{1,\ell}:\mathbb C[S^\ell (M)]_1 \rightarrow  \mathbb C[M]_\ell.
$$ 
Combining $\phi^*_{1,\ell}$ with the restriction homomorphisms induced by the embeddings $X_1\subset \mathbb P(M)$
and $X_\ell \subset \mathbb P(S^\ell (M))$, we get the following 
commutative diagram, for which the rows are isomorphisms and the down arrows are epimorphisms:
$$
\begin{array}{ccc}
\phi^*:\mathbb C[S^\ell (M)]_1 &\rightarrow &\mathbb C[M]_\ell\\
\\
\downarrow\hbox{\small res}_{X_\ell}&\hbox{\rm  } &\downarrow\hbox{\small res}_{X_1}\\
\\
\phi^*:\mathbb C[X_\ell]_1 &\rightarrow &\mathbb C[X_1]_\ell\\
\end{array},\quad
\begin{array}{ccc}
f&\mapsto &f\circ \phi\\
\downarrow&&\downarrow\\
f\vert_{X_\ell}&\mapsto&(f\circ \phi)\vert_{X_1}=f\vert_{X_\ell}\circ \phi \vert_{X_1}.
\end{array}
$$
The restriction of a linear function
on $S^\ell (M)$ to the affine cone $\hat X_\ell$ vanishes if and only if it vanishes on the linear span
$\langle \hat X_\ell\rangle = M^{\odot \ell}$, and hence $\mathbb C[X_\ell]_1=(M^{\odot \ell})^*$.
It follows that
$$
\mathbb C[X_1]=\bigoplus_{\ell\ge 0} \mathbb C[X_1]_\ell =\bigoplus_{\ell\ge 0} (M^{\odot \ell})^*=\bigoplus_{n\ge 0} R_n(M)=R(M).
$$
\end{proof}

\section{Representations and filtrations}\label{reps}
\subsection{PBW filtration}
Let $\bU$ be a complex algebraic unipotent group acting on a cyclic finite dimensional complex vector space $M$,
so for the nilpotent Lie algebra ${\mathfrak N}=\text{Lie\,}\bU$  and a cyclic vector $v_M$ we have $M=\U({\mathfrak N})v_M$.
The PBW filtration of $\U({\mathfrak N})$ is defined by $\U({\mathfrak N})_s= \mathrm{span}\{x_1\dots x_l,\ l\le s, x_i\in{\mathfrak N}\}$, the associated
graded algebra is the symmetric algebra $S({\mathfrak N})$ over ${\mathfrak N}$. Let $\mathfrak N^a$ be the same vector space as $\mathfrak N$
but endowed with the trivial Lie bracket. Then $S({\mathfrak N})=U(\mathfrak N^a)$ is the enveloping algebra of the abelianized version 
$\mathfrak N^a$ of $\mathfrak N$. The increasing filtration
$$
\U({\mathfrak N})_0=\bC1\subseteq \U({\mathfrak N})_1\subseteq \U({\mathfrak N})_2\subseteq\ldots
$$
of $\U({\mathfrak N})$ defines an induced increasing filtration on $M$
$$
PF_0(M)=\bC v_M\subseteq PF_1(M) \subseteq PF_2(M)\subseteq\ldots,
$$
where 
$$
PF_s(M) = \U({\mathfrak N})_sv_M = \mathrm{span}\{x_1\dots x_lv_M,\ l\le s, x_i\in{\mathfrak N}\}.
$$
The {\it associated graded module}
$$
M^a=\bigoplus_{s\ge 0} PF_s(M)/PF_{s-1}(M)
$$
is naturally endowed with the structure of a
graded $U(\mathfrak N^a)$-module, each element $x\in {\mathfrak N}\setminus\{0\}$ induces an operator of degree $1$ on $M^a$.
We denote by $v_{M^a}$ the image of the cyclic generator $v_M$ in $M^a$. It is clear that
\begin{prop}
$M^a$ is a cyclic $U(\mathfrak N^a)$-module with $v_{M^a}$ as a generator.
\end{prop}
\begin{rem}
The construction of the PBW-degeneration is non-trivial even if the initial algebra ${\mathfrak N}$ is abelian. 
In fact, the $U(\mathfrak N^a)$-module
$M^a$ is graded by non-negative integers and each non-trivial operator from ${\mathfrak N}$ has degree one. So if the initial ${\mathfrak N}$-module $M$ is not graded,
the modules $M^a$ and $M$ are not isomorphic.
\end{rem}

\subsection{Essential monomials}
In this section we follow the approach due to Vinberg (see \cite{V}, \cite{G}). Let $\bU, {\mathfrak N}, M$ and $v_M$ be as above.
We fix an ordered basis of ${\mathfrak N}\supset\{ f_1> \dots >f_N\}$ and let ``$>$" be an induced homogeneous monomial order 
(for example the homogeneous reverse lexicographic order,
the homogeneous lexicographic order, \ldots) on the monomials in $\{f_1,\dots,f_N\}$.
In other words, when we compare two monomials, we first compare their total degree.

To a collection of non-negative integers $p_i$, $i=1,\dots,N$, we attach a vector
\[
v_M(\bp)=\bff^\bp v_M=f_1^{p_1}\dots f_N^{p_N}v_M\in M.
\]
Here and below we denote a multi-exponent $(p_1,\dots,p_N)$ simply by $\bp$,
for example, $v_M=v_M({\mathbf 0})$. The degree of $\bff^\bp$ is denoted by $\vert\bp\vert=p_1+\ldots+p_N$.
The sum of multi-exponents $\bp+\bq=(p_1+q_1,\dots,p_N+q_N)$ is defined componentwise.
Since we have a monomial order, $\bp\ge \bq$ and $\bp'\ge\bq'$ implies $\bp+\bp'\ge \bq+\bq'$.
Here and below we write $\bp\ge \bq$ iff $\bff^\bp\ge \bff^\bq$.

\begin{dfn}
A pair $(M,\bp)$ is said to be {\it essential} if
$$
v_M(\bp)\notin \mathrm{span}\{v_M(\bq):\ \bq<\bp\}.
$$
\end{dfn}
If $(M,\bp)$ is essential, then we say that $\bp$ is an essential multi-exponent, $\bff^\bp$ is an essential
monomial in $M$ and we call the vector $v_M(\bp)$ an essential vector.
\begin{dfn}\label{essentialset}
We denote by ${\rm es}(M)\subset\bZ^N_{\ge 0}$ the set of essential
multi-exponents for the module $M$.
\end{dfn}
\begin{rem}\label{remessential}
Since the chosen order is homogeneous and the PBW filtration is given by the total degree,
if $(M,\bp)$ is essential, then $f^\bp v_M$ is non-zero in the PBW degenerate representation $M^a$.
\end{rem}

We use sometimes the notation $S[f_1,\ldots,f_N]$ for $U(\mathfrak N^a)$ if we want to emphasize
that we have fixed a vector space basis for ${\mathfrak N}$.
We define subspaces $F_{<\bp}(M)\subseteq F_\bp(M)\subseteq M$:
\[
F_{<\bp}(M)=\mathrm{span}\{v_M(\bq):\ \bq <\bp\},\quad F_\bp(M)=\mathrm{span}\{v_M(\bq):\ \bq\le\bp\}.
\]
These subspaces define an increasing filtration on $M$, which is finer than the PBW filtration:
\begin{equation}\label{compareFilt}
PF_{\vert\bp\vert -1}(M)\subseteq F_\bp(M) \subseteq PF_{\vert\bp\vert}(M).
\end{equation}
For $i=1,\ldots,N$ let $\mathbf{e}_i=(0,\ldots,0,1,0,\ldots,0)$, where the only nonzero entry is at the $i$-th place.
Then one has 
\[
f_i (F_\bp(M)/F_{<\bp}(M))\subseteq F_{\bp+\mathbf{e}_i}(M)/F_{<\bp+\mathbf{e}_i}(M).
\]
The inclusion in \eqref{compareFilt} implies
\begin{equation}\label{toricaction}
f_i (F_\bp(M)/PF_{\vert\bp\vert -1}(M))\subseteq F_{\bp+\mathbf{e}_i}(M)/PF_{\vert\bp\vert}(M).
\end{equation}
By construction we have $F_\bp(M)\subseteq F_\bq(M)$ if $\bp<\bq$.
We denote the associated graded space by $M^t$ ($t$ is for toric, this notation will be justified later).
The image of $v_M$ in $M^t$ is denoted by $v_{M^t}$.
The space $M^t$ is
$\bZ_{\ge 0}^N$-graded:
\[
M^t=\bigoplus_{\bp\in \bZ_{\ge 0}^N} M^t(\bp),\ \text{where\ } M^t(\bp)=F_\bp(M)/F_{<\bp}(M).
\]
\begin{prop}
\begin{enumerate}
\item The cyclic $U({\mathfrak N})$-module structure on $M$ induces the structure of a cyclic $U(\mathfrak N^a)$-module on $M^t$.
\item The annihilating ideal of $v_{M^t}\in M^t$ is a monomial ideal.
\item Given $\bp\in \bZ^N_{\ge 0}$, the homogeneous component $M^t(\bp)$ is at most one-dimen\-sional, and $\dim M^t(\bp)=1$ if and only if $(M,\bp)$ is essential.
\item The vectors $v_{M^t}(\bp)$, $\bp\in {\rm es}(M)$, form a basis of $M^t$.
\end{enumerate}
\end{prop}
\begin{proof}
Part {(i)} follows by \eqref{toricaction}. The essential vectors $\{v_M(\bq)=\bff^\bp v_M\mid \bp\in {\rm es}(M)\}$ form a basis of $M$
by construction.
Since $\bff^\bq v_{M^t}=v_{M^t}(\bq)=0$ in $M^t$ for $\bq$ {\bf not} essential, it follows by dimension reason
that $\{v_{M^t}(\bq)=\bff^\bp v_M\mid \bp\in {\rm es}(M)\}$ is a basis of $M^t$ and the
$\{\bff^\bp \mid \bp\not\in {\rm es}(M)\}$ form a basis of the annihilating ideal.

Finally, since any two multi-exponents are comparable, the dimension of $M^t(\bp)$ is at most one,
and it is one if and only if $(M,\bp)$ is essential.
\end{proof}

\begin{rem}
Each operator $f_i$ on $M^t$ is homogeneous with respect to the $\bZ^N$-grading and has
degree $\mathbf{e}_i$.
\end{rem}

The next corollary just summarizes the nice behaviour of the essential vectors with respect to the filtrations:
\begin{cor}\label{essbasis}
The vectors $\{v_{M}(\bp)\mid \bp\in {\rm es}(M)\}\subset M$ as well as
$\{v_{M^a}(\bp)\mid \bp\in {\rm es}(M)\}\subset M^a$, %respectively 
$\{v_{M^t}(\bp)\mid \bp\in {\rm es}(M)\}\subset M^t$ form a basis of the corresponding space.
\end{cor}

\begin{rem}
If ${\mathfrak N}$ is abelian, after fixing a basis $\{f_1,\ldots,f_N\}$ we can identify $U({\mathfrak N})$ with the
symmetric algebra $S[f_1,\ldots,f_N]$, and a cyclic module $M$ is of the form  $S[f_1,\ldots,f_N]/I$,
where $I$ is the annihilating ideal. The general procedure described above gives a degeneration of the ideal  $I$ to
a monomial ideal.
\end{rem}
\begin{rem}
Since the filtration induced by a homogeneous order is a refinement of the PBW filtration, it is easy to see that
even if ${\mathfrak N}$ is not abelian, then
$(M^a)^t$ is isomorphic to $M^t$ as ${\mathfrak N}^a$-modules.
\end{rem}

For two cyclic $\U({\mathfrak N})$-modules $M_1$ and $M_2$ with cyclic generators $v_{M_i}\in M_i$, $i=1,2$,  
we denote by $M_1\odot M_2$ the Cartan component in $M_1\T M_2$, i.e.
\[
M_1\odot M_2=\U({\mathfrak N})(v_{M_1}\T v_{M_2})\subset M_1\T M_2.
\]
\begin{prop}\label{plus}
If $(M_1, \bp)$ and $(M_2, \bq)$ are essential, then $(M_1 \odot M_2, \bp+\bq)$ is essential as well.
\end{prop}
\begin{proof}
We denote $v_{M_1}$ by $v_1$, $v_{M_2}$ by $v_2$ and $v_{M_1\odot M_2}$ by $v_{12}$ (we have $v_{12}=v_1\T v_2$).
Similarly, we set $v_1(\bp)=v_{M_1}(\bp)$, $v_2(\bp)=v_{M_2}(\bp)$, $v_{12}(\bp)=v_{M_1\odot M_2}(\bp)$.
We have to show that
\[
f_1^{p_1+q_1}\dots f_N^{p_N+q_N} v_{12}\notin\mathrm{span}\{v_{12}(\br),\ \br<\bp+\bq\}.
\]
In fact, note that if $\br<\bp+\bq$, then
\begin{equation}\label{less}
\bff^\br(v_1\T v_2)\in M_1\T\mathrm{span}\{v_2(\bq'):\ \bq'<\bq\} + \mathrm{span}\{v_1(\bp'):\ \bp'<\bp\}\T M_2
\end{equation}
(acting by a monomial in the $f_i$'s on the tensor product $v_1\T v_2$ means we simply distribute the factors among
$v_1$ and $v_2$). However, $v_{12}(\bp+\bq)$ does not belong to the right hand side of \eqref{less}.
To prove this, it suffices to show that
\begin{equation}\label{sum1}
f_1^{p_1+q_1}\dots f_N^{p_N+q_N} (v_1\T v_2) = C\cdot f_1^{p_1}\dots f_N^{p_N} v_1\T f_1^{q_1}\dots f_N^{q_N}v_2 + \text{ rest },
\end{equation}
where $C$ is a non-zero constant and the remaining terms {\it rest} belong to
\begin{equation}\label{rest}
M_1\T\mathrm{span}\{v_2(\bq'):\ \bq'<\bq\} + \mathrm{span}\{v_1(\bp'):\ \bp'<\bp\}\T M_2.
\end{equation}
Recall that $f_i$ acts as $f_i\T 1 + 1\T f_i$.
The left hand side of \eqref{sum1} is a sum of many terms, among which there are (possibly) several of the
following form:
\[
f_1^{p_1}\dots f_N^{p_N} v_1\T f_1^{q_1}\dots f_N^{q_N}v_2.
\]
Note that while distributing $f_i$ as $f_i\T 1$ and $1\T f_i$ we do not care about the order because we
assume the order on the monomials to be homogeneous -- and hence we can assume that all $f_i$'s commute
because the additional terms coming up during the reordering process are automatically elements of \eqref{rest}.
Now consider the terms
\[
f_1^{p'_1}\dots f_N^{p'_N} v_1\T f_1^{q'_1}\dots f_N^{q'_N}v_2,
\]
where some  $p'_i$ differ from $p_i$. Then we have necessarily
\[
\text{ either } (p_1,\dots,p_N) > (p'_1,\dots,p'_N) \text{ or } (q_1,\dots,q_N) > (q'_1,\dots,q'_N),
\]
since $\bp+\bq=\bp'+\bq'$ (i.e. $p_i+q_i=p'_i+q'_i$ for $i=1,\ldots,N$).
\end{proof}
The proposition above implies that the set $\Gamma_{\bU}(M):=\bigcup_{n\ge 1}(n,\text{es}(M^{\odot n}))\subset \bZ\times\bZ^N$
is naturally endowed with the structure of a semigroup.
\begin{definition}
The semigroup $\Gamma_{\bU}(M)$ is called the {\it essential semigroup} of $M$.
\end{definition}

\section{Coordinate rings and flag varieties}\label{cr}
\subsection{The unipotent case}\label{ringandrep}
Let $\bU$ be a complex algebraic unipotent group acting on a cyclic finite dimensional complex vector space $M$,
so for the nilpotent Lie algebra ${\mathfrak N}=\text{Lie\,}\bU$  and a cyclic vector $v_M$ we have $M=\U({\mathfrak N})v_M$.
We define the ${\bU}$-flag variety $\Fl_{\bU}(M)$ in $ \bP(M)$ as
the closure  of the $\bU$-orbit through the line $\bC v_M$ inside the projective space $\bP(M)$:
\begin{equation}\label{highorbit}
\Fl_{\bU}(M)=\overline{\bU\cdot \bC v_M}\subseteq \bP(M).
\end{equation}
\begin{example}\label{exG}
Let $M=V(\la)$ be a finite dimensional highest weight representation of a simple Lie algebra $\g$ and
let $\g=\fn^-\oplus\fh \oplus \fn^+$ be the Cartan decomposition.
Let $\bU=U^-$ be the maximal unipotent subgroup with Lie algebra $\fn^-$ of the
corresponding Lie group $G$, then the orbit closure $\Fl_{U^-}(M)$ is the (possibly partial) flag variety $G/P_\la$, where $P_\la$
is the parabolic subgroup stabilizing the highest weight line. 
\end{example}

By Proposition~\ref{repdescription} we have the following representation theoretic description of the
homogeneous coordinate ring of the embedded variety $\Fl_{\bU}(M)\subset \bP(M)$:
$$
\mathbb C[\Fl_{\bU}(M)]=
\bigoplus_{\ell\ge 0} \mathbb C[\Fl_{\bU}(M)]_\ell 
=\bigoplus_{\ell\ge 0} (M^{\odot \ell})^*=\bigoplus_{n\ge 0} R_n(M)=R(M).
$$
\subsection{Abelianized version}
Instead of starting with $\bU, {\mathfrak N}$ and $M$ in Section~\ref{ringandrep} we can start
with the abelian Lie algebra ${\mathfrak N}^a$, the module $M^a$ and
$\bU^a=\exp({\mathfrak N}^a)\subset \text{GL}(M^a)$, the Lie group associated to ${\mathfrak N}^a$.
We call the orbit closure
$$
\Fl_{\bU^a}(M^a)=\overline{\bU^a\cdot\bC v_M}\subseteq \bP(M^a)
$$
the PBW-degeneration of $\Fl_{\bU}(M)$.
We have the following representation theoretic description of the
homogeneous coordinate ring of the embedded variety: 
$$
\mathbb C[\Fl_{\bU^a}(M^a)]=R(M^a).
$$
\subsection{Toric version}\label{degenerations}
Instead of starting with $\bU, {\mathfrak N}$ and $M$ in Section~\ref{ringandrep} we can start
with $\bU^a, {\mathfrak N}^a$ and the module $M^t$. We call the orbit closure
$$
\Fl_{\bU^a}(M^t)=\overline{\bU^a\cdot\bC v_M}\subseteq \bP(M^t)
$$
the toric degeneration of $\Fl_{\bU}(M)$. Again 
we have the following representation theoretic description of the
homogeneous coordinate ring of the embedded variety:
$$
\mathbb C[\Fl_{\bU^a}(M^t)]=R(M^t).
$$
We are interested to find conditions which ensure that the abelian respectively toric 
degeneration of $\Fl_{\bU}(M)$ are obtained by a flat degeneration. A first step 
is to describe the structure of $\Fl_{\bU^a}(M^t)$ as a toric variety.

\section{The structure of \texorpdfstring{$\Fl_{\bU^a}(M^t)$}{the closure of the highest weight orbit} as a toric variety}\label{toricsection}
\subsection{Polytopes}\label{polytopes}
A {\it convex lattice polytope} is a polytope $P$ in a Euclidean space ${\mathbb R}^m$
which is the convex hull of finitely many points in the integer lattice ${\mathbb Z}^m \subset {\mathbb R}^m$.
A convex lattice polytope $P$ is called {\it normal} if it has the following property: given any positive integer $n$,
every lattice point of the dilation $nP$, obtained from P by scaling its vertices by the factor n and taking the
convex hull of the resulting points, can be written as the sum of exactly $n$ lattice points in P. Another way
of formulating this property is: the set of lattice points in $nP$ is the $n$-fold Minkowski sum of the lattice points
in $P$ (recall that the Minkowski sum of two subsets $A+B$ is the set of all possible sums $a+b$,
$a\in A$, $b\in B$).

\subsection{Toric varieties}\label{toricvariety}
Let us fix some notation (see \cite{CLS}, Section 2). Let $S\subset\bZ^N$ be a finite
set, $S=(\bs^1,\dots,\bs^k)$. For $\bz=(z_1,\dots,z_N)\in\bC^N$ we set $\bz^\bs=\prod_{i=1}^N z_i^{s_i}$.
{The variety $X(S)\subset\bP^{k-1}$ is defined as the closure of the set
\[
\{(\bz^{\bs^1}:\bz^{\bs^2}:\dots:\bz^{\bs^k})\mid \bz\in(\bC^*)^N\}\subset \bP^{k-1}.
\]
The variety $X(S)$ is a toric variety, it admits a dense orbit by the torus $T=(\bC^*)^N$,
which acts by scaling the variables $z_i$.}

In general the homogeneous coordinate ring of $X(S)$ is the semigroup algebra of
the graded semigroup in $\bN\times\bZ^N$ generated by $\{1\}\times S$. 
Now assume that $S$ is the set of lattice points inside a normal polytope $P$.
Let us consider the polyhedral cone $C(S)$ consisting of elements of  the
form $(n,\bs)$, $\bs\in nP$. The set of lattice points in $C(S)$ is equal to the set $(n,\bs)$, $\bs\in nS$.
Clearly, this set forms a semigroup. We denote by $R(S)$ the complex group algebra of this semigroup. We have $R(S)=\bigoplus_{n\ge 0} R_n(S)$
and the dimension of $R_n(S)$ is given by the cardinality of $nS$.
The ring $R(S)$ is the homogeneous coordinate ring of the projective variety $X(S)$.

\subsection{The toric degeneration \texorpdfstring{$\Fl_{\bU^a}(M^t)$}{} of  \texorpdfstring{$\Fl_{\bU}(M)$}{the flag variety}}
We make the same assumptions and we use the same notation as in Section~\ref{cr}.
We use the notation $\bG_a$ for the one dimensional algebraic group $(\mathbb C,+)$. Recall that 
a commutative unipotent group is isomorphic to a product of several copies of  $\bG_a$.
\begin{prop}\label{prop-both}
The variety $\Fl_{\bU^a}(M^t)$ is both a $\bG_a^N$-variety and a toric variety.
The toric variety $\Fl_{\bU^a}(M^t)$ is isomorphic to $X({\rm es}(M))$.
\end{prop}
\begin{proof}
We need to prove that there exists a torus acting on $\Fl_{\bU^a}(M^t)$ with an open orbit. Consider the $N$-dimensional torus
$T=(\bC^*)^N$ acting on $\bC[f_1,\dots,f_N]$ by
\[
(t_1,\dots,t_N)\cdot f_1^{p_1}\dots f_N^{p_N}=\big(\prod_{i=1}^N t_i^{p_i}\big)\bff^\bp.
\]
Since $M^t=\bC[f_1,\dots,f_N]/I$ for some monomial ideal $I$, we obtain a $T$-action on $M^t$ and hence on $\bP(M^t)$.
We are left to show that $T$ acts on $\bU�\cdot v_M=\bG_a^N\cdot\bC v_M$ with an open orbit.
Let ${\rm es}(M)\subset \bZ_{\ge 0}^N$ be the set of essential multi-exponents for $M$. Then in $M^t$ one has
\[
\exp(\sum_{i=1}^N a_if_i) v_{M^t}=\sum_{\bp\in {\rm es}(M)} \frac{1}{\prod_{i=1}^N p_i!}a_1^{p_1}\dots a_N^{p_N} v_{M^t}(\bp),\ a_i\in\bC.
\]
Since $(t_1,\dots,t_N)\cdot v_{M^t}(\bp)=\prod_{i=1}^N t_i^{p_i} v_{M^t}(\bp)$, we obtain
\[
(t_1,\dots,t_N)\cdot \exp(\sum_{i=1}^N a_if_i) v_{M^t} = \exp(\sum_{i=1}^N a_it_if_i) v_{M^t}.
\]
{ Fix the basis $\{ m_\bp=\frac{1}{\prod_{i=1}^N p_i!} v_{M^t}(\bp)\mid \bp\in \text{es}(M)\}$ of $M^t$ and
let $x_0\in\bP(M)$ be the point $x_0=[\exp(\sum_{i=1}^N f_i) v_{M^t}]$. Now $\{ \exp(\sum_{i=1}^N z_i f_i)\mid z_i\in\bC^*\}$
is an open and dense subset of $\mathbb U^a\simeq \bG_a^N$, and $x_0\in \mathbb U^a [v_{M^t}(\bp)]$. It follows that
$$
\begin{array}{rcl}
\Fl_{\bU^a}(M^t)&=& \overline{\mathbb U^a.[v_{M^t}]}\\
&=&\overline{\{\exp(\sum_{i=1}^N z_if_i).[v_{M^t}]\mid z_1,\ldots,z_N\in\bC^*\}}\\
&=&\overline{\{(z^{\bp^1}:z^{\bp^2}:\ldots:z^{\bp^N})\mid z=( z_1,\ldots,z_N)\in(\bC^*)^N\}},
\end{array}
$$
where $\text{es}(M)=\{\bp^1,\ldots\bp^N\}$, which proves the proposition.}
\end{proof}

\section{A basis of the homogeneous coordinate ring of \texorpdfstring{$\Fl_{\bU}(M)$}{the flag variety}}\label{hcr1}
Let $\bU$ be a complex algebraic unipotent group with Lie algebra ${\mathfrak N}=\text{Lie\,}\bU$.
Let $\bU$ act on a cyclic finite dimensional complex vector space $M$ with cyclic vector $v_M$.
We consider the basis $\{v_M(\bp) \mid \bp \in es(M) \}$ of $M$ and denote the elements of the dual basis in $M^*$ by $\{\xi_\bp \mid \bp \in es(M) \}$.
\begin{lem}\label{<}
Let $\bq=(q_i)_{i=1}^N$ be a multi-exponent (not necessarily essential). Then for any essential $\bp$ such that $\bq<\bp$ we have
$\xi_\bp(v_M(\bq))=0$.
\end{lem}
\begin{proof}
{
The vector $v_M(\bq)$ can be expressed as a linear combination of vectors $v_M(\bq')$ with $\bq'$ essential
and $\bq'\le \bq<\bp$, which proves
the lemma.}
\end{proof}

Recall the description of the homogeneous coordinate ring 
$$
\mathbb C[\Fl_{\bU}(M)]=R(M)=\bigoplus_{n\ge 0} (M^{\odot n})^*.
$$ 
Consider the structure constants $c_{\bp,\bq}^\br$, defined by
\[
\xi_\bp\xi_\bq=\sum_{\br\in {\rm es}(M\odot M)} c_{\bp,\bq}^\br\xi_\br,\quad \bp,\bq\in {\rm es}(M).
\]

\begin{cor}\label{sc}
The structure constant $c_{\bp,\bq}^{\br}$ vanishes if $\br<\bp+\bq$, but $c_{\bp,\bq}^{\bp+\bq}$ does not vanish.
\end{cor}
\begin{proof}
We have
\[
c_{\bp,\bq}^{\br}=(\xi_{\bp}\T\xi_\bq)(\bff^\br(v_M\T v_M))=\sum_{\br'+\br''=\br} 
d_{\br',\br''}\xi_\bp(v_M(\br'))\xi_\bq(v_M(\br'')),
\]
where $d_{\br',\br''}$ are some nonvanishing constants (multiplicities of the corresponding terms).
Now if $\br<\bp+\bq$, then either $\br'<\bp$ or $\br''<\bq$ and by Lemma \ref{<}
we are done. If $\br=\bp+\bq$ and $\br'\ne \bp$, then again either $\br'<\bp$ and $\xi_\bp(v_M(\br'))=0$, or $\br'>\bp$
and then $\br''<\bq=\br-\bp$, so $\xi_\bp(v_M(\br''))=0$. Hence only the terms with $\br'=\bp$, $\br''=\bq$
contribute to $c_{\bp,\bq}^{\bp+\bq}$.
\end{proof}

\begin{lem}\label{renorm}
We can renormalize $\xi_\bp$ in such a way that $c_{\bp,\bq}^{\bp+\bq}=1$ for all essential $\bp$ and $\bq$.
\end{lem}
\begin{proof}
We note that $c_{\bp,\bq}^{\bp+\bq}=\prod_{i=1}^N \frac{(p_i+q_i)!}{p_i!q_i!}$. In fact, according to the proof of Corollary \ref{sc},
$c_{\bp,\bq}^{\bp+\bq}$ is equal to the product over all $i$ of the coefficients of $f_i^{p_i} \T f_i^{q_i}$ in $(f_i\T 1+ 1\T f_i)^{p_i+q_i}$.
Now the desired renormalization is simply $\xi_\bp\to \xi_\bp\prod_{i=1}^N p_i!$.
\end{proof}

\section{Abelianization and tensor product}\label{hcr2}
Let $\bU$ be a complex algebraic unipotent group with Lie algebra ${\mathfrak N}=\text{Lie\,}\bU$.
Let $\bU$ act on two cyclic finite dimensional complex vector spaces $V$ and $W$.
\begin{lem}
There exists a surjective 
homomorphism of $\mathfrak N^a$ modules
\[
(V\odot W)^a\to V^a\odot W^a.
\]
\end{lem}
\begin{proof}
Let $v\in V$ and $w\in W$ be the cyclic vectors and let $V_0\subset V_1\subset\dots\subset V$ and 
$W_0\subset W_1\subset\dots\subset W$ be the induced PBW filtrations on $V$ and $W$ respectively.
We also denote by $ \U({\mathfrak N})_s(v\otimes w)=(V\odot W)_s\subset V\odot W$ the $s$-th space of the induced PBW filtration on
$V\odot W$. Then by definition
\[
(V\odot W)^a= \bigoplus_{s\ge 0} \frac{(V\odot W)_s}{(V\odot W)_{s-1}}.
\]
We have an obvious embedding $(V\odot W)_s=  \U({\mathfrak N})_s(v\otimes w) \hookrightarrow \sum_{i+j=s} V_i\T W_j$ inducing a
natural homomorphism of $\mathfrak N^a$ modules
\begin{equation}
\Psi: \bigoplus_{s\ge 0}\frac{(V\odot W)_s}{(V\odot W)_{s-1}}\to \bigoplus_{s\ge 0}\frac{\sum_{i+j=s} V_i\T W_j}{\sum_{p+q=s-1} V_p\T W_q}.
\end{equation}
Since $(V\odot W)^a$ is cyclic, $\text{Im\,}\Psi$ is generated by $v\T w\in V_0\T W_0$. 
Given $k,l$ such that $k+l=s$, the natural linear map
$$
\tilde\Phi_{k,l}:V_k\otimes W_l \rightarrow \frac{\sum_{i+j=s} V_i\T W_j}{\sum_{p+q=s-1} V_p\T W_q}, \ u\T u'\mapsto \overline{u\T u'},
$$
has kernel $V_k\otimes W_{l-1} + V_{k-1}\otimes W_l$ and hence gives rise to an injective map:
$$
\Phi_{k,l}:\frac{V_k}{V_{k-1}}\T \frac{W_l}{W_{l-1}} \rightarrow \frac{\sum_{i+j=s} V_i\T W_j}{\sum_{p+q=s-1} V_p\T W_q}, \ \bar u\T \bar u'\mapsto \overline{u\T u'}.
$$
Combining these maps for all $k+l=s$, we get a natural isomorphism of  vector spaces:
$$
\Phi_s:\bigoplus_{i+j=s}\frac{V_i}{V_{i-1}}\T \frac{W_j}{W_{j-1}}\rightarrow \frac{\sum_{i+j=s} V_i\T W_j}{\sum_{p+q=s-1} V_p\T W_q},
$$
which gives rise to an isomorphism of $\mathfrak N^a$ modules:
\begin{equation}\label{B}
\Phi: \bigoplus_{s\ge 0} \bigoplus_{i+j=s}\frac{V_i}{V_{i-1}}\T \frac{W_j}{W_{j-1}}\simeq  
\bigoplus_{s\ge 0} \frac{\sum_{i+j=s} V_i\T W_j}{\sum_{i+j=s-1} V_i\T W_j},
\end{equation}
sending the tensor product of classes of two vectors to the class of their tensor product.
Now the composition $\Phi^{-1}\Psi$ gives the desired surjective homomorphism, since the $\mathfrak N^a$ submodule
of the left hand side of \eqref{B}, generated by the product of the cyclic vectors, is equal to
$V^a\odot W^a$.	 
\end{proof}

%Set $M_s:=\U({\mathfrak N})_sv_M$, then $M^{\otimes n+m}$ is filtered by the subspaces
%$$
%(M^{\otimes n+m})_s=\sum_{s=i_1+\ldots +i_{n+m} } M_{i_1}\otimes\cdots\otimes M_{i_{n+m}}
%$$
%and we have a natural map as well as a graded isomorphism (set $(M^{\otimes n+m})_{-1}=0$):
%$$
%(M^{\odot n+m})^a\rightarrow \bigoplus_{s\ge 0} (M^{\otimes n+m})_s/(M^{\otimes n+m})_{s-1}\simeq
%\bigoplus_{s\ge 0}\bigg(\bigoplus_{k+\ell=s} (M^{\odot n})^a_k\otimes (M^{\odot m})^a_\ell\bigg)
%$$
%inducing a natural map $(M^{\odot n+m})^a\rightarrow (M^{\odot n})^a\otimes (M^{\odot m})^a$.
In the setting of sections~\ref{reps} and~\ref{cr} one gets:
\begin{cor}\label{abeliansurj}
\begin{itemize}
\item[{\it i)}]
For any two positive integers $n$ and $m$ there is a natural map 
$(M^{\odot n+m})^a\rightarrow (M^{\odot n})^a\otimes (M^{\odot m})^a$.
\item[{\it ii)}] For any $n$ there exists a natural surjective map $(M^{\odot n})^a\to (M^a)^{\odot n}$.
\end{itemize}
\end{cor}

So we can attach two rings to $\Fl_{\bU^a}(M^a)$: its homogeneous coordinate ring
\[
\mathbb C[\Fl_{\bU^a}(M^a)]=R(M^a)=\bigoplus_{n\ge 0} ((M^a)^{\odot n})^*\text{ and } R^a(M)=\bigoplus_{n\ge 0} ((M^{\odot n})^a)^*.
\]
It is natural to compare these rings and ask under which conditions these are isomorphic.

\begin{example}
In the settings of Example \ref{exG} the orbit closure $\Fl_{\bU^a}(M^a)$ is the PBW-degeneration of flag varieties.
The rings $R(M^a)$ and $R^a(M)$ turn out to be isomorphic in types $A$, $C$ and $G_2$ (see \cite{F1},\cite{FFiL} and
section \ref{flags}).
\end{example}

\section{Favourable modules}\label{favourablemodule}
Let $\bU$ be a complex algebraic unipotent group with Lie algebra ${\mathfrak N}=\text{Lie\,}\bU$,
and let $M$ be a finite dimensional cyclic $\bU$-module with cyclic generator $v_M$.
By Proposition \ref{plus} we know that
\begin{equation}\label{minkowskiessential}
{\rm es}(M^{\odot n})\supseteq \underbrace{{\rm es}(M)+\dots + {\rm es}(M)}_n,
\end{equation}
we are interested in the case where we have equality for all $n$, or, to put it differently, in the case
when the essential semigroup $\Gamma_{\bU}(M)$ of $M$ is generated by $(1,\text{es}(M))$.

\begin{dfn}\label{singleconeredfn}
We say that a finite dimensional cyclic $\bU$-module $M$ is favourable if there exists an
ordered basis $f_1,\dots,f_N$ of ${\mathfrak N}$ and an induced homogeneous monomial order on the PBW basis
such that
\begin{itemize}
\item There exists a  normal polytope $P(M)\subset \bR^{N}$
such that ${\rm es}(M)$ is exactly the set $S(M)$
of lattice points in $P(M)$.
\item $\forall\, n\in\bN:\,\dim M^{\odot n}=\sharp nS(M)$.
\end{itemize}
\end{dfn}

\begin{rem}\label{singleconeremark}
Since $P(M)$ is normal and $S(M)={\rm es}(M)$, we know that $nS(M)$ is the $n$-fold Minkowski sum
of ${\rm es}(M)$. Since $\sharp {\rm es}(M^{\odot n})=\dim M^{\odot n}$, the two conditions in
the definition above ensure that we have equality in \eqref{minkowskiessential} for all $n\ge 1$.
\end{rem}
\begin{rem}\label{singleconeremark1}
The normality of the polytope $P(M)$ depends
on the choice of the induced homogeneous monomial order!
So the property of the module $M$ to be favourable strongly
depends on the choice of the basis and the orderings.
\end{rem}

\begin{prop}\label{propabelian}
If $M$ is a finite dimensional favourable module, then $(M^{\odot n})^a\simeq (M^a)^{\odot n}$ as $S({\mathfrak N})$-modules for all $n\ge 0$.
In particular, the rings $R(M^a)$ and $R^a(M)$ coincide.
\end{prop}
\begin{proof}
We have a natural {surjective} map $(M^{\odot n})^a\to (M^a)^{\odot n}$ of $S({\mathfrak N})$-modules by Corollary~\ref{abeliansurj}.
In the favourable situation the dimensions of the modules coincide and hence they are isomorphic.
\end{proof}
\begin{cor}
If $M$ is favourable, then one can naturally identify the two essential semigroups $\Gamma_{\bU}(M)$ and $\Gamma_{{\bU}^a}(M^a)$.
\end{cor}
\begin{cor}
If $M$ is favourable, then $M^{\odot n}$ is favourable for all $n\ge 1$.
\end{cor}
\begin{proof}
Since $M$ is favorable, the Minkowski sum of $m$ copies of ${\rm es}(M^{\odot n})$ coincides with
${\rm es}(M^{\odot mn})={\rm es}((M^{\odot n})^{\odot m})$,
the polytope $P(M^{\odot n}):=nP(M)\subset \bR^{N}$ is obviously normal, and
the set ${\rm es}(M^{\odot n})$ is the set of lattice points in $P(M^{\odot n})$.
\end{proof}
\section{Flat degenerations}\label{flatdegeneration}
Let $\bU$ be a complex algebraic unipotent group acting on a cyclic finite dimensional complex vector space $M$,
so for the nilpotent Lie algebra ${\mathfrak N}=\text{Lie\,}\bU$  and a cyclic vector $v_M$ we have $M=\U({\mathfrak N})v_M$.
\begin{thm}\label{main}
Let $M$ be a favourable module.
\begin{itemize}
\item[{\it i)}] There exists a flat degeneration of the affine cone $\hat{\Fl}_{\bU}(M)$ into
the affine cone $\hat{\Fl}_{\bU^a}(M^a)$, and for both there exists a flat degeneration
into $\hat{\Fl}_{\bU^a}(M^t)$. 
\item[{\it ii)}] There exists a flat degeneration of ${\Fl}_{\bU}(M)$ into
${\Fl}_{\bU^a}(M^a)$, and for both there exists a flat degeneration
into ${\Fl}_{\bU^a}(M^t)$.
\end{itemize}
The corresponding flat families are equipped with a $\bC^*$-action such that the projection onto $\mathbb A^1$ is equivariant
with respect to the $t^{-1}$-multiplication action on $\mathbb A^1$.
\end{thm}
\begin{rem}
Using the connection with Newton-Okounkov polytopes proved in Section~\ref{NObody}, the degenerations
into the toric variety can also be deduced from \cite{A}. Nevertheless, we state below a full proof because
a slight variation immediately implies also the flat degeneration of ${\Fl}_{\bU}(M)$ in the
PBW-degenerate variety ${\Fl}_{\bU^a}(M^a)$.
\end{rem}
\begin{proof}
We adapt the arguments in \cite{AB}, Proposition 2.2,  respectively \cite{C}, 3.2,
and define a decreasing filtration on the coordinate ring $R(M)$: given $\bp\in\bN^N$, set
\[
R_n(M)^{>\bp}=\bigoplus_{\substack{\bq> \bp\\ \bq\in {\rm es}(M^{\odot n})}} \bC\xi_\bq\ \text{and}\
R_n(M)^{\ge\bp}=\bigoplus_{\substack{\bq\ge \bp\\ \bq\in {\rm es}(M^{\odot n})}} \bC\xi_\bq,
\]
and
\[
R(M)^{>\bp}=\bigoplus_{n\ge 0} R_n(M)^{>\bp}, \ \text{and}\
R(M)^{\ge\bp}=\bigoplus_{n\ge 0} R_n(M)^{\ge\bp}.
\]
Corollary \ref{sc} implies that $R(M)^{\ge\bp}$ and  $R(M)^{>\bp}$ are ideals in
$R(M)$, let $\rm{gr\,} R(M)$ be the associated graded ring:
\[
\rm{gr\,} R(M)=\bigoplus_{n\ge 0} \bigg(\bigoplus_{\bp\in {\rm es}(M^{\odot n})} \frac{R_n(M)^{\ge\bp}}{R_n(M)^{>\bp}}\bigg).
\]
Since the structure constants $c_{\bp,\bq}^{\bp+\bq}$ can be fixed as those after renormalization, we conclude that
$\rm{gr\,} R(M)$ is the coordinate ring of the toric variety defined by ${\rm es}(M)$, i.e.
$\rm{gr\,} R(M)$ is the $\bC$-algebra of the essential semigroup $\Gamma_{\bU}(M)$.
For $\br=(\br',n)\in \Gamma_{\bU}(M)$ we just write $\xi_\br$ for the corresponding element
$\xi_{\br'}\in(M^{\odot n})^*$.

By assumption, ${\rm es}(M)\times 1$ is a minimal set of generators for $\Gamma_{\bU}(M)$,
the corresponding elements $ \xi_{\br^1},\ldots, \xi_{\br^\ell}\in M^*$ generate $R(M)$ by Proposition~\ref{repdescription}.
Let $\mathcal S$ be the polynomial ring $\bC[x_1,\ldots,x_\ell]$. We call the usual grading of $\mathcal S$ the {\it standard
grading} and define a $\Gamma_{\bU}(M)$-grading on the ring by setting
$\deg_\Gamma x_i=\br^i$. Let $I={\rm Ker}\Psi$ be the kernel of the surjective map of $\Gamma_{\bU}(M)$-graded rings
$$
\Psi:\cS\rightarrow \rm{gr\,} R(M),\quad x_i\mapsto \bar \xi_{\br^i}.
$$
The image of a monomial $x_{i_1}\cdots x_{i_q}$ is $\bar \xi_\bp$, where
$\bp=\br_{i_1}+\ldots +\br_{i_q}\in{\rm es}(M^{\odot q})\times q$. The elements $\bar\xi_\bp$, $\bp\in\Gamma_{\bU}(M)$,
are linearly independent, so $I$ is linearly spanned by binomials  $x_{i_1}\cdots x_{i_q}-x_{j_1}\cdots x_{j_q}$,
where $\br_{i_1}+\ldots +\br_{i_q}=\br_{j_1}+\ldots +\br_{j_q}$. We choose generators $\bar g_1,\ldots, \bar g_m\in \cS$
of the ideal $I$ of this form, i.e. $\bar g_k=x_{i_1}\cdots x_{i_{q_k}}-x_{j_1}\cdots x_{j_{q_k}}$. Let $\bq^k=(\bq'^k,q_k)$ be the
$\Gamma_{\bU}(M)$-degree of $\bar g_k$. Since
$\bar g_k(\bar \xi_{\br^1},\ldots,\bar \xi_{\br^\ell})=0$,
it follows that $\bar g_k(  \xi_{\br^1},\ldots,  \xi_{\br^\ell})\in R(M)^{> \bq'^k}$.
More precisely, by Corollary~\ref{sc},
$$
\bar g_k(  \xi_{\br^1},\ldots,  \xi_{\br^\ell})=\sum_{\substack{{\mathbf t}\in \text{es}(M^{\odot q_k})\\ {\mathbf t}>\bq'^k}}
a_{\mathbf t} \xi_{({\mathbf t}, q_k)}
$$
with possibly non-zero coefficients $a_{\mathbf t}$. Since the $\xi_{\br^1},\ldots,  \xi_{\br^\ell}$ generate $R(M)$,
we can find monomials $g_{k,{\mathbf t}}$ of the same standard degree as $\bar g_k$ such that
$g_{k,{\mathbf t}}(\xi_{\br^1},\ldots,  \xi_{\br^\ell})=\xi_{({\mathbf t}, q_k)}$ plus a sum of elements $\xi_{({\mathbf s}, q_k)}$,
${\mathbf s}\in \text{es}(M^{\odot q_k})$,
such that ${\mathbf s}>{\mathbf t}$. Since $\text{es}(M^{\odot q_k})$ is a finite set, this implies that we can find a polynomial
\begin{equation}\label{toricdecomp}
g_k=\bar g_k+\sum_{j=1}^{r_k} g_{k,j}
\end{equation}
such that $g_k$ is homogeneous of standard degree $q_k$,
each $g_{k,j}$ is homogeneous of $\Gamma_{\bU}(M)$-degree
$\bq^{k,j}=(\bq'^{k,j},q_k)$ such that $\bq'^{k,j}>\bq'^k$, and
$$
g_k( \xi_{\br^1},\ldots, \xi_{\br^\ell})=0.
$$
To deal with the degeneration into the abelianized flag variety, one uses as above the $\xi_\bp$ to define a filtration of 
$R(M)$ with respect to the total PBW-degree, the associated graded ring $\text{gr}^a\, R(M)$ is
the homogeneous coordinate ring of the PBW-degenerate flag variety $\Fl_{\bU^a}(M^a)$
(Proposition~\ref{propabelian}).
Rewrite the sum $\sum_{j=1}^{r_i}g_{i,j}$ above
as $\sum_{j=0}^{t_i} h_{i,j}$, where $h_{i,j}$ is a sum of $\Gamma_{\bU}(M)$-homogeneous elements such that the total PBW-degree
of $h_{i,j}$ is equal to $(\text{total PBW-degree of}\ \bar g_i)+j$. Set $g_i^a= \bar g_i +h_{i,0}$.
% and let $I^a$ be the ideal in $\cS$ generated by $g_1^a,\ldots,g_m^a$.
\begin{lem}\label{zwischenlemma}
The natural surjective maps below are isomorphisms:  
$$
\begin{array}{rl}
\text{\it a)}&\Phi: \cS/(g_1\cS + \ldots +g_m\cS)\rightarrow R(M), \quad x_i\mapsto \xi_{\br^i},\\
\text{\it b)}&\Phi^a: \cS/(g_1^a\cS + \ldots +g^a_m\cS)\rightarrow \text{gr}^a\, R(M), \quad x_i\mapsto \xi_{\br^i}.
\end{array}
$$
\end{lem}
\begin{proof} (of the lemma)
To prove that the map $\Phi$ is an isomorphism we define a filtration and show that the associated graded map is injective.
For $\bq\in \bN^N$
let $\cS^{\ge \bq}_n$ be the span of all monomials $x_1^{{a}^1}\ldots x_\ell^{{a}^\ell}$ in $\cS$ of standard degree $n$
such that ${a}^1\br'^1+\ldots+{a}^\ell\br'^\ell\ge\bq$. Then $\cS^{\ge \bq}=\oplus_{n\ge 0} \cS^{\ge \bq}_n$ is obviously an ideal, we
define $\cS^{> \bq}_n$ and $\cS^{> \bq}$ similarly. Let $\text{gr\,}\cS$ be the associated graded
algebra:
$$
\text{gr\,}\cS =\bigoplus_{n\ge 0} \bigg(\bigoplus_{\bp\in {\rm es}(M^{\odot n})} \frac{\cS_n^{\ge\bp}}{\cS_n^{>\bp}}\bigg).
$$
Note that $g_i$ and $\bar g_i$ are representatives in $\cS$ of the same class in $\text{gr\,}\cS$.
Let $p:\cS\rightarrow \cS/(g_1\cS + \ldots +g_m\cS)$ be
the projection. The algebra $\cS/(g_1\cS + \ldots +g_m\cS)$ is filtered by the images $p(\cS^{\ge \bq})$ of the ideals,
let $\text{gr\,}\big(\cS/(g_1\cS + \ldots +g_m\cS)\big)$ be the associated graded
algebra.
The filtration of $R(M)$ induced by the images $\Phi\circ p(\cS^{\ge \bq})$ is exactly the filtration
of $R(M)$ we started with, so we get induced morphisms:
$$
\begin{array}{rcc}
 \text{gr\,}\cS &   &   \\
{\bigg \downarrow}{\text{gr\,}p}&{\searrow}{\text{gr}\Psi}  &   \\
\text{gr\,}\big(\cS/(g_1\cS + \ldots +g_m\cS)\big)  & \mathrel{\mathop{\longrightarrow
}_{\mathrm{gr}\Phi}}  &   \text{gr\,}R(M)
\end{array}.
$$
The classes of $g_i$ and $\bar g_i$ coincide in the associated graded algebra, so we
have a surjective map $\cS/(\bar g_1\cS + \ldots +\bar g_m\cS)\rightarrow \text{gr\,}\big(\cS/(g_1\cS + \ldots +g_m\cS)\big)$.
The isomorphism $\cS/(\bar g_1\cS + \ldots +\bar g_m\cS)\simeq \rm{gr\,} R(M)$,
implies that ${\text{gr}\Phi}$, and hence also
$\Phi$, is injective, and thus $\Phi$ is an isomorphism. The proof for $\Phi^a$ is similar.
\end{proof}\noindent
({\it Continuation of the proof of Theorem~\ref{main}}) 
We consider now first the degeneration of the affine cone $\hat{\Fl}_{\bU}(M)$ into
the affine cone $\hat{\Fl}_{\bU^a}(M^t)$. By the lemma above we know:
\begin{equation}\label{Rpresent}
R(M)=\cS/(g_1\cS + \ldots +g_m\cS)\
\end{equation}
where the generators $g_i$ are homogeneous with respect to the standard grading and
have a decomposition into homogeneous parts for the $\Gamma(M)$-grading
such that
\begin{equation}\label{Relationspresent}
 g_i=\bar g_i +\sum_{j=1}^{r_i} g_{i,j},\quad \text{where\ }\deg g_{i,j}=\bq^{i,j}> \bq^i=\deg \bar g_i.
\end{equation}
The set $\{\bq^i,\bq^{i,j})\mid i=1,\ldots,m,\ j=1,\ldots,r_i \}\subset \bN^N$
is finite, so by \cite{C}, Lemma~3.2, there exists a linear map $e:\bR^N\rightarrow \bR$ such that
$e(\bN^N)\subseteq\bN$ and $e(\bq^i)<e(\bq^{i,j})$. Let
\begin{equation}\label{Ralgebra}
{\mathcal R}=\cS[x_0]/{\mathcal I},
\end{equation}
where ${\mathcal I}$ is the ideal generated by the elements
\begin{equation}\label{RIdeal}
\bar g_i +\sum_{j=1}^{r_i} x_0^{e(\bq^{i,j})-e(\bq^{i})} g_{i,j}
\end{equation}
for $ i=1,\ldots, m$. Let $X$ be the variety
\begin{equation}\label{varietyX}
X=\left\{v=
\left(\begin{array}{c}v_0 \\ v_1 \\ \vdots \\ v_\ell\end{array}\right)\in\bC^{\ell+1}\mid f(v)=0\ \forall f\in {\mathcal I}\right\}.
\end{equation}
The projection $\pi:\bC^{\ell +1}\rightarrow {\mathbb A}^1$, $v\mapsto v_0$ onto the first coordinate induces a
projection (denoted by the same letter)
$$
\pi:X\rightarrow  {\mathbb A}^1.
$$
The construction implies for the
fibre $\pi^{-1}(1)$ that $x_0=1$ and hence $X_1=\pi^{-1}(1)$ is isomorphic to $\hat\Fl_{\bU}(M)$.
Similarly, for $X_0=\pi^{-1}(0)$ we have $x_0=0$ and hence $X_0$
is isomorphic to $\hat\Fl_{\bU^a}(M^t)$, the affine cone over the toric variety.
We define a $\bC^*$-action on $\bC^{\ell+1}$ by
\begin{equation}\label{Cstaraction}
t\cdot \left(\begin{array}{c}v_0 \\ v_1 \\ \vdots \\ v_\ell\end{array}\right)=
\left(\begin{array}{c}t^{-1}v_0 \\ t^{e(\br^1)} v_1 \\ \vdots \\ t^{e(\br^\ell)}v_\ell\end{array}\right).
\end{equation}
Note that
$$
\begin{array}{rl}
(\bar g_i +\sum_{j=1}^{r_i} &x_0^{e(\bq^{i,j})-e(\bq^{i})} g_{i,j})(t\cdot v)\\
=&t^{e(\bq^i)}\bar g_i(v) + \sum_{j=1}^{r_i} t^{(-e(\bq^{i,j})+e(\bq^{i}))}x_0^{e(\bq^{i,j})-e(\bq^{i})}(v) t^{e(\bq^{i,j})}g_{i,j}(v)  \\
=&t^{e(\bq^i)}(\bar g_i +\sum_{j=1}^{r_i} x_0^{e(\bq^{i,j})-e(\bq^{i})} g_{i,j})(v).
\end{array}
$$
As an immediate consequence we see that $X$ is stable under the $\bC^*$-action, and the map $\pi$ is
$\bC^*$-equivariant with respect to the $t^{-1}$-multiplication action of $\bC^*$ on $\bC$.

By the $\bC^*$-action we know that all fibres over a point different from $0$
are isomorphic to $\hat\Fl_{\bU}(M)$, and the special fibre over $0$
is isomorphic to $\hat\Fl_{\bU^a}(M^t)$. It follows
that $X=\overline{ \bC^*\cdot(\pi^{-1}(1))}$ is irreducible (the 
$\overline{ \bC^*\cdot(\pi^{-1}(1))}$ contains the special fiber,
since the dimension of the special fiber coincides with the dimension of the general 
fibers). Since $\pi$ is surjective, it
follows  that $\pi$ is flat (\cite{H}, Chapter III, Proposition 9.7).

Using part {\it b)} of Lemma~\ref{zwischenlemma}, one can proceed as in the toroidal case
to prove the existence of a flat degeneration of the affine cone $\hat{\Fl}_{\bU}(M)$ into
the affine cone $\hat{\Fl}_{\bU^a}(M^a)$.
To prove the existence of a flat degeneration of $\hat{\Fl}_{\bU^a}(M^a)$
into $\hat{\Fl}_{\bU^a}(M^t)$, one proceeds as above, the only difference being
that instead of starting with $g_k$ as in \eqref{toricdecomp} one starts with $g_k^a= \bar g_k +h_{k,0}$
and decomposes $h_{k,0}$ into its $\Gamma_{\bU}(M)$-homogeneous parts.

To prove the second part of the theorem we extend the natural standard grading on $\cS$ to $\cS[x_0]$ by setting $\deg x_0=0$,
the ideal ${\mathcal I}$ (see \eqref{RIdeal}) is homogeneous with respect to this grading, so the algebra ${\mathcal R}=\cS[x_0]/{\mathcal I}$
inherits a natural grading. Due to the canonical isomorphism ${\mathcal R}_0\simeq \bC[x_0]$,
the variety $Y=\text{Proj\,}{\mathcal R}$ comes naturally equipped with a projective morphism
$\theta:Y\rightarrow \mathbb A^1$. By replacing the arguments above by the appropriate ones
for proper maps one proves part {\it ii)} of the theorem.
\end{proof}

\begin{cor}\label{normality}
If $M$ is a finite dimensional favourable module, then $\Fl_{\bU}(M))$, its PBW-degeneration
$\Fl_{\bU^a}(M^a)$ and its toric degeneration $\Fl_{\bU^a}(M^t)$
are projective varieties which for the given embeddings are projectively normal and arithmetically Cohen-Macaulay.
\end{cor}
\begin{proof}
The condition on $M$ to be favourable implies that the polytope $P(M)$ is normal and hence
the variety $\Fl_{\bU^a}(M^t)\simeq X({\rm es}(M))\subseteq \mathbb P(M^t)$ is
projectively normal and arithmetically Cohen-Macaulay (see Theorem 9.2.9 and Exercise 9.2.8 \cite{CLS}),
i.e. the affine cone $\hat\Fl_{\bU^a}(M^t)\subseteq M^t$ over the projective variety is normal and Cohen-Macaulay.

Since one knows now that the special fibre has the claimed good properties, it is a standard fact for flat families
which are trivial away from ``$0$" that all fibres have these nice properties, see for example \cite{knut}.
\end{proof}

\section{Newton-Okounkov bodies and filtrations}\label{NObody}
Our general reference for more details on Newton-Okounkov bodies is \cite{KK}.
Let $\bU$ be a complex algebraic unipotent group acting on a cyclic finite dimensional complex vector space $M$.
Without loss of generality assume the action to be faithful.
Denote by ${\mathfrak N}=\text{Lie\,}\bU$ its Lie algebra, let $v_M\in M$ be a fixed cyclic vector and denote by $\bV$
the stabilizer in $\bU$ of $[v_M]\in\mathbb P(M)$.

The field $\mathbb C(\Fl_{\bU}(M))$ of rational functions on $\Fl_{\bU}(M)$ coincides with the field $\mathbb C(\bU. [v_M]))$ of rational functions on the orbit.
The orbit of a unipotent group is an affine space. To determine the Newton-Okounkov body of $\mathbb C(\Fl_{\bU}(M))$
(associated to a valuation), we want to fix an appropriate parameterization of this affine space.

\begin{example}\label{decomp}
Let the notation be as in Example~\ref{exG}. For a dominant weight $\lambda$ let $P_\lambda$ be the stabilizer of the highest weight line in $V(\lambda)$
and let $P_\lambda^-$ be the opposite parabolic subgroup. Denote by $\bU'$ the unipotent radical of $P_\lambda^-$,
the stabilizer $\bV$ of $[v_\lambda]$ in $\bU=U^-$ is the intersection $\bU\cap L_\lambda$
with the Levi subgroup $L_\lambda\subset P_\lambda$. Then $\bU=\bU' \bV$, $\bU'\cap \bV=\{id\}$,
$\bU. [v_M]=\bU'. [v_M]$, and the stabilizer in $\bU'$ is trivial. It follows that $\bU. [v_M]\simeq \bU'$ as an affine variety.
\end{example}

\subsection{The decomposition case}\label{decompNO}  
Suppose first that $\bU$ admits a subgroup $\bU'$ such that $\bU=\bU'\bV$ and $\bU'\cap\bV=\{id\}$
(as in Example~\ref{decomp} above).
We fix a basis $\mathbb B_{\mathfrak N'}$ of $\mathfrak N'=\text{Lie\,} \bU'$,
a basis $\mathbb B_{\mathbb V}=\{f_{N+1},\ldots,f_r\}$ of $\text{Lie\,} \mathbb V$ and we set
 $\mathbb B_{\mathfrak N}=\mathbb B_{\mathbb \bU'}\cup \mathbb B_{\mathbb V}$.

Fix a homogeneous monomial ordering on the monomials in the elements of $\mathbb B_{\mathfrak N}$.
Since $\bV$ is the stabilizer of $v_M$, it is obvious that $ (M^{\odot n},{\mathbf p})$ is essential only if 
$p_{N+1}=\ldots=p_{r}=0$. The essential semigroup $\Gamma_{\bU}(M)$ is by definition a subset of 
$\bN\times\bZ^r$, but since the last entries are necessarily identically zero for an essential multi-exponent, 
we omit in the following these components and view $\Gamma_{\bU}(M)$ as a subset of $\bN\times\bZ^N$. 
By abuse of notation we write $\bfp\ge \bfq$ 
for $\bp,\bq\in \bN^{N}$ if $\bfp'\ge \bfq'$, where $\bfp',\bfq'\in \bZ^{r}$
are the tuples obtained from $\bfp,\bfq$ by adding zero entries. 

Let $x_{1},\ldots,x_{N}$ be the basis of $(\mathfrak N')^*$ 
dual to the basis $\{f_{1},\ldots, f_{N}\}$ of $\mathfrak N'$.
The exponential map 
$$
\exp: \mathfrak N'\rightarrow \bU',\quad X\mapsto \exp(X),
$$
is an isomorphism of affine varieties. So the field of
rational functions $\bC(\Fl_{\bU}(M))$ can be identified with $\bC(x_{1},\ldots, x_{N})$ and the $x_{j}$,
$j=1,\ldots,N$, form a system of parameters.
We write $\bfx^\bfp$ for a monomial $x_{1}^{p_{1}}\cdots x_{{N}}^{p_{N}}$. 

We define a $\bZ^N$-valued valuation on $\bC(\Fl_{\bU}(M))$ as follows:
given a polynomial $g(\bfx)=\sum a_\bfp \bfx^\bfp$, we define
\begin{equation}\label{valuationdef}
\nu(g(\bfx))=\min\{\bfp \mid a_\bfp\not=0\}.
\end{equation}
For a rational function $h = \frac{g}{g'}$ we define $\nu(h(\bfx))=\nu(g(\bfx))-\nu(g'(\bfx))$.
The valuation $\nu$ is called the {\it lowest term valuation with respect to the
para\-meters $x_{{1}},\ldots, x_{N}$ and the total order ``$\,\ge$''}.

Let $\xi_{\mathbf 0}$ (section~\ref{hcr1}) 
be the dual vector of the fixed cyclic generator $v_M \in M$. Consider the homogeneous coordinate ring
$A=\bC[\Fl_{\bU}(M)]=\bigoplus_{n\ge 0} A_n$ of the embedded variety $\Fl_{\bU}(M)\hookrightarrow \bP(M)$.
We associate to $A$ the valuation semigroup $S_A$ as follows:
\begin{equation}\label{valuationsemigroup}
S_A=S(A,\nu,\xi_{\mathbf 0})=\bigcup_{n\ge 1}\{(n,\nu(\frac{g}{\xi_{\mathbf 0}^n}))\mid g\in A_n-\{0\}\}\subseteq \bN\times\bZ^N.
\end{equation}
The fact that we have a valuation implies that this is a semigroup.
Recall that we view the essential semigroup $\Gamma_{\bU}(M)$ as a subset of $\bN\times\bZ^N$.
\begin{prop}\label{semigleichsemi1}
The essential semigroup and the valuation semigroup coincide:
$$
S_A=\Gamma_{\bU}(M).
$$
\end{prop}
\begin{proof}
Let $(n,\bfp)\in \{n\}\times \text{es}(M^{\odot n})$, we want to evaluate
$$
\xi_\bfp\bigg( \exp(x_{1}f_{1}+\ldots+ x_{N}f_{N}) v_{M^{\odot n}}\bigg).
$$
Since the argument of the exponential is a nilpotent operator, the sum is finite. 
It is necessary to reorder the factors within certain monomials. So
new terms may occur, but they are of lower total degree and hence strictly smaller with respect
to the homogeneous monomial order. So the exponential can be written as a linear combination
of ordered monomials as follows: 
$$
\sum_{k\ge 0} \frac{1}{k!} (x_{\ell_1}f_{\ell_1}+\ldots+ x_{\ell_N}f_{\ell_N})^k v_{M^{\odot n}}=
\sum_{k\ge 0} \bigg(\sum_{\substack{\mathbf q\in \mathbb N^N\\\vert\bq\vert=k} }
(c_\bq\mathbf x^{\mathbf q}\mathbf f^\bq +\sum_{\substack{\bq'\in \mathbb N^N\\ \vert\bq'\vert<k}} a_{\bq',\bq}\mathbf x^{\mathbf q}\mathbf f^{\bq'})v_{M^{\odot n}}\bigg),
$$
where $c_\bq\not=0$.
If ${\mathbf f}^{\mathbf r} v_{M^{\odot n}}$ is not an essential vector, 
then the vector can be rewritten as a linear combination of smaller essential vectors:
$$
\mathbf x^{\mathbf q}{\mathbf f}^{\mathbf r} v_{M^{\odot n}}= 
\sum_{\substack{\mathbf s<\mathbf r \le \bq \\ (n,\mathbf s)\in (n,\text{es}(M^{\odot n}) )}} \mathbf x^{\mathbf q}b_{\mathbf s} {\mathbf f}^{\mathbf s} v_{M^{\odot n}}.
$$
So we can rewrite the sum above as a linear combination of essential vectors:
$$
\exp(x_{\ell_1}f_{\ell_1}+\ldots+ x_{\ell_N}f_{\ell_N}) v_{M^{\odot n}}=
\sum_{(n,\bfq)\in (n,\text{es}(M^{\odot n}))}\bigg(\frac{c_\bfq}{\vert\bq\vert} \mathbf x^\bfq +
\sum_{\mathbf r > \bfq }a'_{\bfq,\mathbf r} \mathbf x^{\mathbf r} \bigg){\mathbf f}^\bfq v_{M^{\odot n}}.
$$
It follows that 
$$
\xi_\bfp\bigg( \exp(x_{\ell_1}f_{\ell_1}+\ldots+ x_{\ell_N}f_{\ell_N}) v_{M^{\odot n}}\bigg)=
\frac{c_\bfp}{\vert\bp\vert} \mathbf x^\bfp +
\sum_{\mathbf r > \bfp }a'_{\bfp,\mathbf r} \mathbf x^{\mathbf r}.
$$
Since the coefficient $c_\bfp\not=0$, we get
$\nu\big(\frac{\xi_\bfp}{\xi_{\mathbf 0}^n}\big)=\bfp$.
It follows that $\Gamma_{\bU}(M)\subseteq S_A$. Since the $\xi_\bfp$ form a basis of
$A$ with pairwise different evaluations, we have equality.
\end{proof}

\subsection{The general case}\label{generalcaseNO}
If we do not have the decomposition as in \ref{decompNO}, then the possible choices for a basis of $\mathfrak N$ are more restrictive.
We fix a sequence of subgroups $\bU=\bU_1\supset\ldots\supset \bU_r\supset \bU_{r+1}=\{id\}$ such that $\bU_{i+1}$ is normal in
$\bU_{i}$ for $i\ge 1$ and of codimension 1. Since $\bU$ is unipotent, such a sequence always exists. 
We get an induced filtration for $\bV= \bV_{i_1}\supset\ldots\supset \bV_{i_s}\supset\bV_{i_{s+1}}=\{id\}$
by subgroups, i.e. for $j=1,\ldots,s$ we have
$$
\bV_{i_j}=\bU_{i_j}\cap\bV=\bU_{i_j+1}\cap\bV=\ldots=\bU_{i_{j+1}-1}\cap\bV\supsetneq  \bV_{i_{j+1}}=\bU_{i_{j+1}}\cap\bV 
$$
and $\hbox{codim}_{\bV_{i_{j}}} \bV_{i_{j+1}}=1$.
Fix a basis $\mathbb B_{\mathfrak N}=\{f_1,\ldots,f_r\}$ of ${\mathfrak N}$ compatible with the filtrations above, so $\{f_i,\ldots,f_r\}$ is a basis
of $\hbox{Lie\,}\bU_i$ and for all $j=1,\ldots,s$, the subset $\{f_{i_j},f_{i_{j+1}},\ldots,f_{i_s}\}$ is a basis of $\hbox{Lie\,}\bV_j$. In particular,
$\mathbb B_{\mathbb V}=\{f_{i_1},,\ldots,f_{i_s}\}$ is a basis of $\hbox{Lie\,}\mathbb V$.
We fix an induced homogeneous monomial order ``$\ge$'' on the monomials in $\mathbb B_{\mathfrak N}$.
\begin{lem}\label{unipotentonepugs}
\begin{itemize}
\item[{\it i)}] $ (M^{\odot n},{\mathbf p})$ is essential only if $p_{i_1}=\ldots=p_{i_s}=0$.
\item[{\it ii)}] Let $\mathbb G(f_i)$ be the subgroup $\{\exp(t f_i)\mid t\in\mathbb C\}\subset \bU$. The product map
$$
m_{\bU}:\mathbb G(f_1)\times \mathbb G(f_2)\times\cdots\times \mathbb G(f_r)\rightarrow \bU
$$
is an isomorphism of affine varieties.
\item[{\it iii)}]  Let $\{f_{\ell_1},\ldots,f_{\ell_N}\}$ be the complement in $\mathbb B_{\mathfrak N}$ of $\mathbb B_{\mathbb V}$. The product map
$$
m_{\bU/\bV}:\mathbb G(f_{\ell_1})\times \mathbb G(f_{\ell_2})\times\cdots\times \mathbb G(f_{\ell_N})\rightarrow \bU /\bV
$$
induces an isomorphism of affine varieties.
\end{itemize}
\end{lem}
\begin{proof} Suppose $\mathbf f^{\bp}$ is such that $p_{i_j}>0$ for some $j\in\{1,\ldots,s\}$. 
By commuting $f_{i_j}$ to the right one can rewrite $\mathbf f^{\mathbf p}$ as a linear combination 
of strictly smaller elements and a monomial that annihilates $v^{\odot n}$, so $ (M^{\odot n},{\mathbf p})$ is not essential.

Let $1\le i\le r$. Since $ \mathbb G(f_i)$ is a subgroup of $\bU_i$, we have a natural map
$$
m_\bU^i: \mathbb G(f_i)\times \mathbb U_{i+1}\rightarrow \mathbb U_i.
$$
Since $\mathbb G(f_i)\cap \mathbb U_{i+1}$ is a finite group and unipotent, the intersection is equal to $\{id\}$
and the map is hence injective. Now $\mathbb U_{i},\mathbb U_{i+1}$ are unipotent and $\mathbb U_{i+1}$ is a normal subgroup in 
$\mathbb U_i$. Using the Zassenhaus formula one shows that the map is also surjective. Both varieties are smooth, so by 
Zariski's main theorem (\cite{Kum}, Theorem A.11) it follows that the map is in fact an isomorphism. Now {\it ii)} follows by induction.

To prove {\it iii)}, note that $\bU_i=\mathbb G(f_i)\bU_{i+1}=\bU_{i+1} \mathbb G(f_i)$ because $\bU_{i+1}$ is normal in $\bU_{i}$.
By applying this argument to the subgroups corresponding to $\bV$, it follows by part {\it ii)} of the lemma:
$$
\bU=\mathbb G(f_1) \mathbb G(f_2)\cdots \mathbb G(f_r)=
\mathbb G(f_{\ell_1}) \cdots \mathbb G(f_{\ell_N})  \mathbb G(f_{i_s})\cdots  \mathbb G(f_{i_1})
$$
By applying the same arguments to $\bV$ as to $\bU$ in {\it ii)} and above, we see 
$$
\bV=\mathbb G(f_{i_1})\cdots  \mathbb G(f_{i_s})=\mathbb G(f_{i_s})\cdots  \mathbb G(f_{i_1})
$$
and hence $\bU=\mathbb G(f_{\ell_1}) \cdots \mathbb G(f_{\ell_N}) \bV$, which finishes the proof of the lemma.
\end{proof}
Let $\{x_{\ell_1},\ldots,x_{\ell_N}\}$ be the basis of $({\mathfrak N}/\hbox{Lie\,}\bV)^*$ dual to the basis
$\{\bar f_{\ell_1},\dots,\bar f_{\ell_N}\}$ of ${\mathfrak N}/\hbox{Lie\,}\bV$.
Since ${\mathbb U}\cdot [v_M]\subset \Fl_{\bU}(M)$ is a smooth open affine subset isomorphic
to $\mathbb G(f_{\ell_1})\times \cdots\times \mathbb G(f_{\ell_N})$, the field of
rational functions $\bC(\Fl_{\bU}(M))$ can be identified with $\bC(x_{\ell_1},\ldots, x_{\ell_N})$ and the $x_{\ell_j}$,
$j=1,\ldots,N$, form a system of parameters.
We write $\bfx^\bfp$ for a monomial $x_{\ell_1}^{p_{\ell_1}}\cdots x_{{\ell_N}}^{p_{\ell_N}}$, where
$\bfp\in\bN^N$. We use the same convention as above: $\bfp\ge \bfq$ if $\bfp'\ge \bfq'$, 
where $\bfp',\bfq'\in \bZ^{r=N+s}$ are the tuples obtained from $\bfp,\bfq$ by adding zero entries in the places $i_1,\ldots,i_s$.
Let $\ge$ also denote the induced monomial order on $\bC[x_{\ell_1},\ldots, x_{\ell_N}]$,
i.e. $\bfx^\bfp\ge \bfx^\bfq$ if and only if $\bfp\ge \bfq$.

Now we define the $\bZ^N$-valued valuation on $\bC(\Fl_{\bU}(M))$ and the valuation semigroup $S_A$
as in \eqref{valuationdef} and \eqref{valuationsemigroup}.
\begin{prop}\label{semigleichsemi2}
The essential semigroup and the valuation semigroup coincide:
$$
S_A=\Gamma_{\bU}(M).
$$
\end{prop}
\begin{proof}
We have to evaluate
$\xi_\bfp\bigg( \exp(x_{\ell_1}f_{\ell_1}) \cdots \exp(x_{\ell_N}f_{\ell_N}) v_{M^{\odot n}}\bigg)$
for an element $(n,\bfp)\in \{n\}\times \text{es}(M^{\odot n})$.
Now
\begin{equation}\label{function1}
\exp(x_{\ell_1}f_{\ell_1}) \cdots \exp(x_{\ell_N}f_{\ell_N}) v_{M^{\odot n}}
=\bigg(\sum_{\bfq\in \bN^N}a_\bfq \mathbf x^\bfq {\mathbf f}^\bfq \bigg)v_{M^{\odot n}}
\end{equation}
for some constants $a_\bfq\not=0$. The same arguments as in the proof of Proposition~\ref{semigleichsemi1}
apply: after rewriting non essential vectors as a linear combination of smaller essential vectors one gets:
$$
\xi_\bfp\bigg(\exp(x_{\ell_1}f_{\ell_1}) \cdots \exp(x_{\ell_N}f_{\ell_N})) v_{M^{\odot n}}\bigg)=
a_\bfp \mathbf x^\bfp +\sum_{\bfp' > \bfp }a'_{\bfp'} \mathbf x^{\bfp'}.
$$
Since the coefficient $a_\bfp\not=0$, we get
$\nu\big(\frac{\xi_\bfp}{\xi_{\mathbf 0}^n}\big)=\bfp$.
It follows that $\Gamma_{\bU}(M)\subseteq S_A$. Since the $\xi_\bfp$ form a basis of
$A$ with pairwise different evaluations, we have equality.
\end{proof}

\subsection{Newton-Okounkov body} 
One associates to $A$ also the cone $C$ generated by $S_A=\Gamma_{\bU}(M)$ in $\bR\times \bR^N$:
$$
C=\text{smallest closed convex cone centered at $0$ containing $S_A$.}
$$
\begin{definition}
The {\it Newton-Okounkov body} $\Delta(\Fl_{\bU}(M),\nu,\ge)$ of $\Fl_{\bU}(M)$ in $\bR^N$ 
is the projection of the intersection $C\cap 1\times \bR^N$ on $\bR^N$.
In other words, the Newton-Okounkov body is the closure of the convex hull of the rescaled exponents:
$$
\Delta(\Fl_{\bU}(M))=\Delta(\Fl_{\bU}(M),\nu,\ge)= \overline{\text{convex}(\bigcup_{{n\ge 1}} \{\frac{\bfp}{n}\mid (n,\bfp)\in S_A\}) }.
$$
\end{definition}
\begin{thm}
Let $\bU$ be a be a complex algebraic unipotent group with Lie algebra $\mathfrak N$ and let $M$
be a finite dimensional cyclic $\bU$-module with cyclic vector $v_M$ and stabilizer $\bV\subset \bU$
of $[v_M]\in \mathbb P(M)$.
Assume that either $\bU$ admits a decomposition $\bU=\bU'\bV$ as in section~\ref{decompNO} or the basis
of $\mathfrak N$ has been chosen as in section~\ref{generalcaseNO}.
If $M$ is favourable and $P(M)$ is the associated
lattice polytope (Definition~\ref{singleconeredfn}), then we have for the Newton-Okounkov bodies:
$$
\Delta(\Fl_{\bU}(M))=P(M)=\Delta(\Fl_{\bU^a}(M^a)).
$$
\end{thm}
\section{The multicone version}\label{multiconeversion}
Let $G$ be a simple simply connected complex algebraic group $G$.
We fix a Cartan decomposition of its Lie algebra $\fg=\fn^+\oplus\fh\oplus\fn^-$, an ordered basis
$\underline\fn^-:=\{f_1,\ldots,f_N\}$ of $\fn^-$, and we fix a homogeneous monomial order on the PBW basis.
Let $\bU\subset G$ be the maximal unipotent subgroup with Lie algebra $\fn^-$.

Let $\omega_1,\ldots, \omega_n$ be the fundamental weights for $G$.
For a dominant integral weight $\la=a_1\omega_1+\ldots +a_n\omega_n$
we denote by the support $\sup \la$ the set of fundamental weights
\begin{equation}\label{support}
\sup\la=\{\omega_i\mid a_i\not=0\}.
\end{equation}

By Proposition \ref{plus} we know that
\begin{equation}\label{minkowskiessential2}
{\rm es}(V(\la))\supseteq \underbrace{{\rm es}(V(\omega_1))+\dots+{\rm es}(V(\omega_1))}_{a_1}+\dots +
\underbrace{{\rm es}(V(\omega_n))+\dots+{\rm es}(V(\omega_n))}_{a_n},
\end{equation}
we are interested in the case when we have equality for all dominant weights:

\begin{dfn}
The pair $(G,\underline\fn^-)$ is called {\it favourable} for the fixed order if
\begin{itemize}
\item for each fundamental weight $\omega_i$ there exists a normal polytope $P_i$
such that the lattice points $S_i\subset P_i$ index the essential monomials for $V(\omega_i)$,
\item for a dominant weight $\lambda=a_1\omega_1+\ldots+a_n\omega_n$ let $P_\lambda:= a_1P_1+...+a_nP_n$
be the corresponding Minkowski sum of the polytopes $P_i$. Let $S_\lambda\subset P_\lambda$ be the set of lattice
points. Then:
$$
\forall\ \text{dominant weights\,}\la:\ \dim V(\lambda) = \sharp S_\lambda=\sharp(\sum_{i=1}^n a_iS_i).
$$
\end{itemize}
\end{dfn}
\begin{rem}
As in Definition~\ref{singleconeredfn}, the conditions are split into two parts of rather different kind:
the first deals with the structure of the fundamental representations as $\fn^-$-modules,
the second is of more combinatorial nature comparing dimension formulas for representations
with formulas counting lattice points in polytopes.

The conditions imply equality in \eqref{minkowskiessential2} (compare Remark~\ref{singleconeremark}):
$$
\sharp {\rm es}(V(\la))= \dim V(\lambda) = \sharp\left(\sum_{i=1}^n  a_iS(\omega_i)\right) = \sharp(\sum_{i=1}^n a_i{\rm es}(V(\omega_i)))
$$
which, by the inclusion in \eqref{minkowskiessential2} is only possible if
$$
{\rm es}(V(\la))= \underbrace{{\rm es}(V(\omega_1))+\dots+{\rm es}(V(\omega_1))}_{a_1}+\dots +
\underbrace{{\rm es}(V(\omega_n))+\dots+{\rm es}(V(\omega_n))}_{a_n}.
$$
\end{rem}
\begin{rem}
The condition for $(G,\underline\fn^-)$ to be {favourable} implies that the global Okounkov body (see for example \cite{Gon}) 
for $G/B$ is a polyhedral cone,
and one gets an explicit description of the cone. More precisely, let
$$
GOB(G/B)=\{(p,\lambda)\in\mathbb R^N\times X_{\mathbb R}\mid \lambda=\sum a_i\omega_i, a_1,\ldots,a_n\in \mathbb R_{\ge 0},
p\in\sum a_iP_i\},
$$
where $X_{\mathbb R}$ is the real span of the weight lattice. It is easy to see that this is a convex cone such that if $\lambda$
is a regular dominant integral weight, then the fibre $\pi^{-1}(\lambda)$ of the projection map 
$$
\Psi: GOB(G/B)\hookrightarrow \mathbb R^N\times X_{\mathbb R}\longrightarrow X_{\mathbb R}
$$
is the polytope $P_\lambda$, so $GOB(G/B)$ is the global Okounkov body for $G/B$. Consider the vertices  $\{p_j^i \mid j=1,\ldots,r_i\}$ 
of the polytope $P_i$ for $i=1,\ldots,n$. A simple calculation shows that $GOB(G/B)$ is the cone spanned by 
$$
\{(p_j^i,\omega_i)\mid i=1,\ldots,n;j=1,\ldots,r_i\}.
$$
\end{rem}
This strong condition has some beautiful consequences:
\begin{thm}\label{theorem218}
If $(G,\underline\fn^-)$ is favourable, then
\begin{itemize}
\item[{\it i)}] $V({\la+\mu})^a\simeq V({\la})^a\odot V({\mu})^a$ as $S(n^-)$-modules,
\item[{\it ii)}] the representations $V(\la)$, $\la$ a dominant weight, are favourable for $\fn^-$,
\item[{\it iii)}] for all dominant integral weights $\la$, { the projective varieties $\Fl_{\bU^a}(V^a_\la)\subseteq \bP(V^a_\la)$
and $\Fl_{\bU^a}(V^t_\la)\subseteq \bP(V^t_\la)$ are projectively normal and arithmetically Cohen-Macaulay,}
\item[{\it iv)}] there are embeddings $\Fl_{\bU^a}(V(\la+\mu)^a)\hookrightarrow \Fl_{\bU^a}(V(\la)^a)\times \Fl_{\bU^a}(V(\mu)^a)$,
\item[{\it v)}] the variety $\Fl_{\bU^a}(V(\la)^a)$ depends on the support $\sup\la$ of $\la$ only (see \eqref{support}),
\item[{\it vi)}] the vectors $\{v_{\lambda}(\bp)\mid \bp\in S(\lambda)\}$ form a basis for $V(\lambda)$,
$V(\la)^a$ respectively $V(\lambda)^t$ (depending on whether one chooses $v_{\lambda}(0)$ to be the cyclic generator
of $V(\lambda)$, $V(\lambda)^a$ or $V(\la)^t$).
\end{itemize}
\end{thm}
\proof
We have a natural surjective map $V({\lambda+\mu})^a\rightarrow V(\lambda)^a\odot V(\mu)^a$.
Since $(G,\underline\fn^-)$ is favourable, one has
$$
\text{es}(V({\la+\mu}))=S(\la+\mu)=S(\lambda)+S(\mu)=\text{es}(V({\la}))+\text{es}(V({\mu})),
$$
so by Proposition~\ref{plus} this map is injective and hence an isomorphism. The Minkowski sum of lattice polytopes
is a lattice polytope, so the condition being favourable
implies that all the polytopes $P(\la)$ are normal. As an immediate consequence
we see: all representations $V(\la)$ are favourable and the varieties
$\Fl_{\bU^a}(V(\la)^a)$ and $\Fl_{\bU^a}(V(\la)^t)$ are
projectively normal and arithmetically Cohen-Macaulay
(Corollary~\ref{normality}). Similarly, part {\it vi)} is a consequence of Corollary~\ref{essbasis}.

The Segre embedding $\bP(V(\la)^a)\times \bP(V(\mu)^a)\hookrightarrow \bP(V(\la)^a\otimes V(\mu)^a)$
and the isomorphism $V({\lambda+\mu})^a\simeq V(\lambda)^a\odot V(\mu)^a$ implies that the
image of $\Fl_{\bU^a}(V({\la+\mu})^a)$ in $\bP(V(\la)^a\otimes V(\mu)^a)$ lies in the embedded
product $\Fl_{\bU^a}(V(\la)^a)\times \Fl_{\bU^a}(V(\mu)^a)$. By embedding $\Fl_{\bU^a}(V(\la)^a)$
in the corresponding product of degenerate flag varieties for fundamental weights,
it is easy to see that $\Fl_{\bU^a}(V(\la)^a)\simeq \Fl_{\bU^a}(V(\nu)^a)$ for
$\nu=\sum_{\omega\in \sup \la}\omega $, and hence  $\Fl_{\bU^a}(V(\la)^a)$ depends only on the support of $\la$.
\endproof
\begin{rem}
In the next section we will see that $(G,\underline\fn^-)$ is favourable for type ${\tt A}_n$, ${\tt C}_n$ and ${\tt G}_2$.
\end{rem}

\section{The classical examples}\label{flags}
In this section we illustrate the construction of the previous section on the example of flag varieties
of classical groups.
Let $\g$ be a simple Lie algebra with the Cartan decomposition $\g=\fn\oplus\fh\oplus\fn^-$
and let $G$ be the corresponding semisimple, simply connected complex algebraic group.
As before, $\bU$ denotes the maximal unipotent subgroup with Lie algebra $\fn^-$.
Let $\triangle_+$ be the set of positive roots of $\g$ and let $\al_1,\dots,\al_n$ be the simple roots.
Let $f_\beta\in\fn^-$, $\beta\in\triangle_+$, be a root basis of $\fn^-$.

Let $\la$ be a dominant integral weight for the Lie algebra $\g$ and let $V(\la)$ be the corresponding irreducible $\g$-module of
highest weight $\la$. Fix a highest weight vector $v_\la\in V(\la)$; in particular, $\fn v_\la=0$ and $V(\la)=\U(\fn^-) v_\la$.
We will be interested in the degenerate modules $V(\la)^a$ and $V(\la)^t$ introduced above. To apply Theorem~\ref{theorem218}
we need to introduce an ordering $\beta_1,\dots,\beta_N$ of the positive roots and fix a homogeneous monomial order.
Then the set of essential monomials is fixed and we give a combinatorial description in terms of a normal polytope.\\
In the following we will consider only orderings having the following special property (we give examples of such orderings below):\\
Let ``$\succ$'' be the standard partial order on the set of positive roots. We assume that
\[
\beta_i\succ \beta_j \text{ implies } i<j.
\]
An ordering with this property (the larger roots come first) will be called a {\it good ordering}.
Once we fix such a good ordering, this induces an ordering on the basis vectors $f_\beta$.
As monomial order on the PBW basis we fix the induced homogeneous reverse lexicographic order.
\begin{example}\label{goodA}
In type $A_n$ ($\g=\msl_{n+1}$) the positive roots are of the form $\al_{i,j}=\al_i+\dots +\al_j$ for $1\le i\le j\le n$.
Here is an example of a good ordering in type $A_n$:
\begin{gather*}
\beta_1=\al_{1,n},\\
\beta_2=\al_{1,n-1},\ \beta_3=\al_{2,n},\\
\dots,\\
\beta_{(n-1)n/2+1}=\al_1,\dots, \beta_{n(n+1)/2}=\al_n.
\end{gather*}
\end{example}

\begin{example}
In type $C_n$ ($\g=\msp_{2n}$) the positive roots are of the form
\begin{gather*}
\al_{i,j}=\al_i+\al_{i+1}+\dots +\al_j,\ 1\le i\le j\le n,\\
\alpha_{i, \ol{j}} = \alpha_i + \alpha_{i+1} + \ldots +
\alpha_n + \alpha_{n-1} + \ldots + \alpha_j, \ 1\le i\le j\le n
\end{gather*}
(note that $\al_{i,n}=\al_{i,\ol n}$).
Here is an example of a good ordering in type $C_n$:
\begin{gather*}
\beta_1=\al_{1,\ol{1}},\\
\beta_2=\al_{1,\ol{2}},\ \beta_3=\al_{2,\ol{2}},\\
\dots,\\
\beta_{(n-1)n/2+1}=\al_{1,\ol{n}},\dots, \beta_{n(n+1)/2}=\al_{n,\ol{n}},\\
\beta_{n(n+1)/2+1}=\al_{1,n-1},\dots, \beta_{n(n+1)/2+n-1}=\al_{n-1,n-1},\\
\dots,\\
\beta_{n^2}=\al_{1,1}.
\end{gather*}
\end{example}

We now recall the polytopes describing the basis of the PBW graded modules in types $A$ and $C$ (\cite{FFoL1}, \cite{FFoL2}, \cite{FFoL3})
and the basis in type $G_2$ from \cite{G}. We just write $P(\la)$ instead of $P(V(\la))$.

\subsection{Type \texorpdfstring{${\tt A}_n$}{An}}\label{typeA}
Let $\om_1,\dots,\om_n$ be the fundamental weights, we denote $f_{\al_{i,j}}$ by $f_{i,j}$.
Note that $N=\dim\fn^-=n(n+1)/2$.

For a dominant integral weight $\la=\sum_{k=1}^n m_k\om_k$, $m_k\in\bZ_{\ge 0}$,
we define a polytope $P(\lambda)\subset \bR^N_{\ge 0}$ and the set $S(\la)\subset \bZ^N_{\ge 0}$ as follows:

A sequence ${\bf b}=(\beta_1,\dots,\beta_r)$ of positive roots is called a {\it Dyck path} if the first and the last roots are
simple roots ($\beta_1=\al_{i,i}$, $\beta_r=\al_{j,j}$, $i\le j$), and if $\beta_m=\al_{p,q}$, then $\beta_{m+1}=\al_{p+1,q}$ or
$\beta_{m+1}=\al_{p,q+1}$.

\begin{dfn}
The polytope $P(\lambda)\subset \bR^N_{\ge 0}$ is defined as the set of
points $\bp=(p_\beta)_{\beta\in\triangle_+}$ in $\bR^N_{\ge 0}$ satisfying the following
inequalities (with integer coefficients):
for all Dyck paths ${\bf b}$ with $\beta_1=\al_{i,i}$, $\beta_r=\al_{j,j}$ one has
\[
p_{\beta_1}+p_{\beta_2}+\dots +p_{\beta_r}\le m_i+\dots +m_j.
\]
\end{dfn}
The set $S(\la)=P(\la)\cap \bZ^N_{\ge 0}$ is the set of lattice points in $P(\la)$.
We proved in \cite{FFoL1} that $\{\bff^\bp v_\la\mid \bp\in S(\la)\}$
forms a basis of $V(\la)^a$ and hence of $V(\la)$ itself.

\subsection{Type \texorpdfstring{${\tt C}_n$}{Cn}}
Let $\om_1,\dots,\om_n$ be the fundamental weights,
we will use the following abbreviations for the roots and the operators:
$$
\al_i = \al_{i,i},\ \al_{\ol{i}} = \al_{i,\ol{i}},\ f_{i,j}=f_{\al_{i,j}},\
f_{i,\ol j}=f_{\al_{i,\ol j}}.
$$
Note that $N=\dim\fn^-=\#\triangle_+=n^2$.
We recall the usual order on the alphabet $A = \{1, \ldots, n, \ol{n-1}, \ldots, \ol{1}\}$
$$ 1 <2 < \ldots < n-1 < n < \ol{n-1} < \ldots < \ol{1}.$$
A {\it symplectic Dyck path} is a sequence ${\bf b}=(\beta_1, \dots, \beta_r)$ of positive roots such that:
the first root is a simple root, $\beta_1=\al_{i,i}$; the last root is either a simple root $\beta_r= \al_j$ or
$p(k) = \al_{\ol{j}}$ ($i \le j \leq n$); if $\beta_m=\al_{r,q}$ with $r, q \in A$ then $\beta_{m+1}$ is
either $\al_{r,q+1}$ or $\al_{r+1,q}$, where $x+1$ denotes the smallest element in $A$ which is bigger than $x$.

\begin{dfn}
The polytope $P(\lambda)\subset \bR^N_{\ge 0}$ is defined as the set of
points $\bp=(p_\beta)_{\beta\in\triangle_+}$ in $\bR^N_{\ge 0}$ satisfying the following
inequalities (with integer coefficients): for all Dyck paths
${\bf b}$ with $\beta_1=\al_{i,i}$, $\beta_r=\al_{j,j}$ one has
\[
p_{\beta_1}+p_{\beta_2}+\dots +p_{\beta_r}\le m_i+\dots +m_j;
\]
for all Dyck paths ${\bf b}$ with $\beta_1=\al_{i,i}$, $\beta_r=\al_{j,\ol{j}}$ one has
\[
p_{\beta_1}+p_{\beta_2}+\dots +p_{\beta_r}\le m_i+\dots +m_n.
\]
\end{dfn}
The set $S(\la)=P(\la)\cap \bZ^N_{\ge 0}$ is the set of lattice points in $P(\la)$.
We proved in \cite{FFoL2} that the set $\{\bff^\bp v_\la\mid \bp\in S(\la)\}$
forms a basis of $V(\la)^a$ and hence of $V(\la)$ itself.

\subsection{Type \texorpdfstring{${\tt G}_2$}{G2}}
Let $\al_1,\al_2$ be simple roots. The six positive roots are as follows:
\[
\beta_1=3\al_1+2\al_2,\ \beta_2=3\al_1+\al_2,\ \beta_3=2\al_1+\al_2,\ \beta_4=\al_1+\al_2,\ \beta_5=\al_2,\ \beta_6=\al_1.
\]
We note that this ordering is good.
Let $\la=k\omega_1+l\omega_2$, $k,l\ge 0$.

\begin{dfn}
The polytope $P(\lambda)\subset \bR^6_{\ge 0}$ is defined as the set of
points $\bp=(p_\beta)_{\beta\in\triangle_+}$ in $\bR^6_{\ge 0}$ satisfying the following
inequalities (with integer coefficients):
\begin{gather*}
p_5\le l,\ p_6\le k,\\
p_2+p_3 + p_6\le k+l,\quad p_3+p_4+p_6\le k+l,\quad p_4+p_5+p_6\le k+l,\\
p_1+p_2+p_3+p_4+p_5\le k+2l,\qquad p_2+p_3+p_4+p_5+p_6\le k+2l.
\end{gather*}
\end{dfn}
The set $S(\la)=P(\la)\cap \bZ^6_{\ge 0}$ is the set of lattice points in $P(\la)$.
It is proved in \cite{G} that the set $\{\bff^\bp v_\la\mid \bp\in S(\la)\}$
is a basis of $V(\la)^t$ and hence of $V(\la)^a$
and $V(\la)$.

\subsection{Essential sets and \texorpdfstring{$S(\la)$}{the lattice points in the polytope}}
To apply Theorem~\ref{theorem218} we need to prove that the ordering gives a favourable pair.
So we need to show that $P(\lambda)$ is the Minkowski sum of the polytopes for the fundamental weights and furthermore $P(\lambda)$ is normal.
We recall the following proposition (\cite{FFoL3}, \cite{G}):
\begin{prop}\label{classicalminkowski}
Let $\g$ be of type $A_n$, $C_n$ or $G_2$. Then for any two dominant weights $\la$ and $\mu$ one has $S(\la+\mu)=S(\la)+S(\mu)$.
\end{prop}

\begin{lem}\label{polyacg}
The polytopes $P(\la)$ defined above for $\fg$ of type $\tt A_n$, $\tt C_n$ or $\tt G_2$ are normal.
\end{lem}

\begin{proof}
The polytopes are defined by inequalities with integer coefficients, hence the vertices
have rational coordinates. Let now $v\in P(\lambda)$ be a point with rational coordinates.
Fix $q\in \bN$ such that $qv$ has integral coordinates, so $qv\in S(q\la)$.
By Proposition~\ref{classicalminkowski}, $qv$ is an element of
the $q$-fold Minkowski sum of $S(\lambda)$, so one can write
$qv= s_1 + s_2 + \ldots + s_q$, where $s_1, s_2, \ldots, s_q \in S(\lambda)$,
and hence $v= \frac{1}{q} s_1 + \ldots \frac{1}{q} s_q$ is in the convex hull of the lattice points of
$P(\lambda)$. It follows that $P(\lambda)$ is a lattice polytope, which is normal  by Proposition~\ref{classicalminkowski}.
\end{proof}
\begin{thm}
Let $\g$ be of type $A_n$, $C_n$ or $G_2$. Assume that the positive roots are ordered and the ordering is good.
Then the pair $(G,\underline\fn^-)$ is favourable.
\end{thm}
\begin{proof}
By Proposition~\ref{classicalminkowski} and Lemma~\ref{polyacg}, it remains to show that the set ${\rm es}(V(\la))$ coincides with $S(\la)$.
The case of $G_2$ is worked out in \cite{G}, where it is proved that $S(\la)$ indexes a basis of $V(\la)^t$.

Let $\g=\msl_n$. First, we prove the theorem for fundamental weights $\la=\omega_k$.
Then $V(\la)=\Lambda^k(V({\omega_1}))$ and $V({\omega_1})$ is the $n$-dimensional vector representation.
Fix the standard basis $w_1,\dots,w_n$ of $V({\omega_1})$. We denote by $w_{i_1,\dots,i_k}$ the wedge product $w_{i_1}\wedge\dots\wedge w_{i_k}$.
For example, $v_\la=w_{1,2,\dots,k}$. Let
\[
I=\left\{i_1,\dots,i_k \right\} \; \text{ with } \ 1\le i_1<\dots <i_s\le k <i_{s+1}<\dots <i_k\le n.
\]
Then we set $w_I := w_{i_1}\wedge\dots\wedge w_{i_k}$ and note that $w_I \neq 0$ in $V(\lambda)$. 
Further, let $J=\{1,\dots,k\}\setminus \{i_1,\dots,i_s\}$. We write $J=(j_1,\dots,j_{k-s})$, where $k < j_1<\dots <j_{k-s} \leq n $. There might be several
multi-exponents $\bp$ such that $v_M(\bp)=w_I$. We claim that the minimal monomial (and hence the essential one) is
\begin{equation}\label{monom}
f_{j_1,i_k-1}f_{j_2,i_{k-1}-1}\dots f_{j_{k-s},i_{s+1}-1}v_\la
\end{equation}
(recall that $f_{i,j}w_i=w_{j+1}$).
In fact, first of all the minimal length of a monomial is exactly $k-s$. Now a monomial $\bff^\bp$ such that $\bff^\bp v_\la$ is proportional to
$w_I$ is of the form
\begin{equation}\label{vect}
f_{j_{\sigma(1)},i_k-1}f_{j_{\sigma(2)},i_{k-1}-1}\dots f_{j_{\sigma(k-s)},i_{s+1}-1}v_\la
\end{equation}
for some permutation $\sigma\in S_{k-s}$. We claim that the minimal monomial \eqref{vect} corresponds to $\sigma=\rm{id}$. 
In fact, the minimal root vector among all $f_{j_\sigma(\ell), i_{\ell}-1}$ is $f_{j_1,i_k-1}$ 
(see the ordering in Example \ref{goodA}). \\
This implies that there is a factor $f_{j_1,i_k-1}$  in the minimal monomial with $f^\bp v_\lambda=w_I$. 
By proceeding in the same way we obtain the claim by downwards induction. 
It suffices now to note that \eqref{monom} belongs to $S(\omega_k)$. We thus obtain the inclusion
$es(V(\omega_k)) \subseteq S(\omega_k)$. Since the cardinalities of these sets coincide, we have the equality.

Similarly we check that for fundamental weights of the symplectic algebras ${\rm es}(V({\omega_k}))=S(\omega_k)$.

Now let us consider the general $\la$. Thanks to Proposition \ref{plus} we know that for dominant weights  $\la$ and $\mu$
we have ${\rm es}(V(\la))+{\rm es}(V(\mu))\subset {\rm es}(V({\la+\mu}))$. But
\[
\sharp({\rm es}(V(\la))+{\rm es}(V(\mu)))=\sharp (S(\la)+ S(\mu))=\sharp S(\la+\mu)=\dim V(\la+\mu)= \sharp {\rm es}(V(\la+\mu))
\]
(here, the third equality is proved in \cite{FFoL1}). We conclude that the equalities ${\rm es}(V({\omega_k}))=S(\omega_k)$ for $k=1,\ldots,n$ imply
${\rm es}(V(\la))=S(\la)$ for any dominant weight $\la$.

The proof in the case of $\g=\msp_{2n}$ is similar, it suffices to consider the fundamental weights.
\end{proof}
As in Section~\ref{NObody} we deduce:
\begin{cor}
For all dominant weights $\la$,
there exists an appropriate evaluation on the field $\bC(\bU.[v_\lambda])$ respectively $\bC(\bU^a.[v_\lambda])$, such that
% we get the following
%interpretation of the polytope in terms of Newton-Okounkov bodies:
the polytope $P(\la)$ is the Newton-Okounkov body of $\Fl_{\bU}(V(\la))$ and of $\Fl_{\bU^a}(V(\la)^a)$.
\end{cor}

\subsection{Automorphism group}\label{auto-grou}
Let us turn again to the $G = SL_{n+1}$-case and regular $\lambda$, and use that we have an explicit description of the 
polytope $P(\lambda)$ in terms of inequalities defined by Dyck paths. Then $P(\lambda) \subset \mathbb{R}_{\geq 0}^{N}$ where $N = n(n+1)/2$, 
we denote $F \subset \mathbb{R}^N$ the associated normal fan. The rays $F(1)$ of $F$ are given by
\begin{enumerate}
\item $\mathbb R. e_{i,j}$, $1\le i\le j \le n$,
\item $\mathbb R.(-\sum_{k=1}^r e_{i_k,j_k})$ for all possible Dyck paths $(i_1,j_1),\ldots, (i_r,j_r)$.
\end{enumerate}
Now following \cite{Nil, Cox}, we define the Demazure roots of the toric variety
\[
\mathcal R=\{ m\in \mathbb{Z}^N \mid  \exists \tau\in F(1)\hbox{ \rm such that\ } \langle v_\tau,m\rangle=-1, \langle v_\eta,m\rangle\ge 0\ \forall \eta\in F(1),\eta\not=\tau \}.
\]
It is easy to see that if $\tau$ is a ray of type $(ii)$, then there is a unique Demazure root for this $\tau$, 
namely $m$ with $m_{1,n} = 1$ and $m_{i,j} = 0$ else and this is the unique semisimple root in $\mathcal R$.\\
For $\tau = \mathbb Re_{k,l}$, one has the following Demazure roots: 
\begin{itemize}
\item For $l=1 \; : \; m=(m_{i,j}) \hbox{\rm\ where\ }\ m_{k,1}=-1\hbox{\rm\  and\ }\exists k'>k: m_{k',1}=1\hbox{\rm\ and\ }m_{i,j}=0\ \forall (i,j)\not=(k,1),(k',1)
$,
\item for $k=n \; : \; m=(m_{i,j}) \hbox{\rm\ where\ }\ m_{n,l}=-1\hbox{\rm\  and\ }\exists l'>l: m_{n,l'}=1\hbox{\rm\ and\ }m_{i,j}=0\ \forall (i,j)\not=(n,k),(n,k').
$,
\item and else $m=(m_{i,j}) \hbox{\rm\ where\ }\ m_{k,l}=-1\hbox{\rm\  and\ }m_{i,j}=0\ \forall (i,j)\not=(k,l)$.
\end{itemize}
Then the number of Demazure roots is exactly
\[
1+\frac{1}{2}n(n+1)+\frac{1}{2}n(n-1)+\frac{1}{2}n(n-1)= \frac{3}{2}n^2-\frac{1}{2}n+1.
\]  
We conclude that the connected component of the automorphism group of the toric variety is the semidirect product of a reductive group with a $N$ dimensional torus, a semisimple
part isomorphic to $SL_2$ or $PSL_2$, and a $\frac{3}{2}n^2-\frac{1}{2}n-1$-dimensional unipotent radical.
\subsection{More examples}
For $G = SL_3$, the toric variety obtained in this way is isomorphic to the one
constructed in \cite{AB} via the Gelfand-Tsetlin polytope  (see also \cite{KM,GL}), 
but even for $G = SL_4$ the two are not isomorphic in general. To be more precise, it has been shown in \cite{BF, Fou2} that for $G=SL_n$, 
the toric varieties associated to the polytopes $P(\lambda)$ described in this paper 
and the Gelfand-Tsetlin-polytopes are isomorphic if and only if $\lambda$ is supported only on the first 
two nodes or on the last two nodes or on the first and last node. Recall that in \cite{ABS} a bijection between the 
integral points in the two polytopes has been provided. It would be quite interesting to have a geometric interpretation of this bijection. 
\\
Next take for $G = SL_4$ the total order on the positive roots obtained by the following reduced decomposition
$w_0 = s_2s_1s_3s_2s_1s_3$ of the longest Weyl group element (see \cite{AB, Lit}). In this case
our toric variety is isomorphic to the one in \cite{AB} associated to this reduced decomposition 
(for most of these computations we used the program \textit{polymake} (\cite{GJ}). One can show, again using \textit{polymake}, that for  $G = SL_6$, 
that there is no reduced decomposition of the longest Weyl group element, such that for regular $\lambda$ 
our toric variety is isomorphic to the one obtained from the corresponding string polytope.\\
Our construction of the polytope and the toric variety is very explicit, for example we are able to 
compute the rays of the associated normal fan, the Demazure roots and the automorphism group of the toric variety 
(for $G =SL_n$ and $\lambda$ regular) (see Section~\ref{auto-grou}).\\
Some other special cases are investigated in \cite{BD}. 
Let $\omega$ be a fundamental weight for $G$ such that $\langle\omega,\theta^\vee\rangle = 1$,
where $\theta$ is the highest root (this includes all minuscule and cominuscule fundamental weights).
\begin{coro}
There exists an ordering of the positive roots and a homogeneous
ordering on the monomials in $U(\fn^-)$ such that for all $m\ge 1:$
$V(m\omega)$ is favourable for $\bU$ with a highest weight vector as cyclic generator.
\end{coro}

Instead of working with the full $G$-representation it is natural to look also at special 
submodules stable under a maximal unipotent
subgroup $\bU$ of $G$, the standard example being Demazure submodules. One can show for $G=SL_n$ and $\lambda=m\omega$ a multiple of a fundamental
weight that all Demazure submodules are favourable $\bU$-modules (\cite{BF}). 
For arbitrary $\lambda$ a partial answer was provided in \cite{Fou}, while the general question remains open even in the $SL_n$-case. 

\section*{Acknowledgments}
We are grateful to Dmitry Timashev for his comments on the Vinberg construction and 
for providing us with the reference \cite{G}.

E.F. was partially supported by the Dynasty Foundation.
The article was supported within the framework of a subsidy granted to the HSE by the Government of the Russian Federation 
for the implementation of the Global Competitiveness Program
and within the framework of the Academic Fund Program at the National Research University Higher School of Economics (HSE) in 
2015--2016 (grant № 15-01-0024).

G.F. and P.L. were partially supported by the DFG Priority Program SPP 1388 on Representation Theory.

\end{document}